\documentclass[9pt,a4paper]{article}

\usepackage{amsmath,amsthm,amssymb}
\usepackage{latexsym}
\usepackage{enumerate}
\usepackage[latin1]{inputenc}
\usepackage{mathptmx}       
\usepackage{helvet}         
\usepackage{courier}        
\usepackage{amsfonts}
\usepackage{amscd}
\usepackage{graphicx}
 
\usepackage{amstext}

\usepackage{color}
\usepackage{graphicx}

\newcommand{\field}[1]{\mathbb{#1}}

\newcommand{\E}{\field{E}}

\newcommand{\at}[2]{\genfrac{}{}{0pt}{}{#1}{#2}}

\newtheorem{theo}{Theorem}

\newtheorem{lemma}{Lemma}
\newtheorem{lemm}{Lemma}[section]

\newtheorem{defi}[lemm]{Definition}
\newtheorem{prop}[lemm]{Proposition}

\newtheorem{remm}[lemm]{Remmark}

\newtheorem{exam}{Example}
\newtheorem{hyp}{Assumption}

\let\p=\pi

\let\F=\Phi

\newcommand{\Prob}{\mathbb{P}}
\newcommand{\R}{\mathbb{R}}

\newcommand{\N}{\mathbb{N}}

\newcommand{\1}{\textbf{1}}
\DeclareMathOperator*{\supess}{\textrm{ess}\,\sup}

\title{Thermodynamic Formalism on the Skorokhod space:  the continuous-time Ruelle operator,  entropy, pressure, entropy production and  expansiveness}

\author{J. Knorst, A. O. Lopes, G. Muller, and A. Neumann \\  Instituto de Matem\'atica e Estat\'istica, UFRGS, Porto Alegre, Brasil}

\begin{document}

\maketitle


\begin{abstract}

\footnotesize{Consider the semi-flow given by the continuous time shift
$\Theta_t:\mathcal{D} \to \mathcal{D} $, $t \geq 0$,
acting on the
 $\mathcal{D} $ of \textit{c\`{a}dl\`{a}g} paths (right continuous with left limits) $w: [0,\infty) \to S^1$, where $S^1$ is the unitary circle (one can also take $[0,1]$ instead of $S^1$). We equip the space $\mathcal{D} $ with the Skorokhod metric, and we show that the semi-flow is expanding. We also introduce a stochastic semi-group $e^{t\, L}$, $t \geq 0,$ where $L$ (the infinitesimal generator) acts linearly on continuous functions $f:S^1\to\mathbb{R}$. This stochastic semigroup and an initial vector of probability $\pi$ define an associated stationary shift-invariant probability $\mathbb{P}$ on the Polish space $\mathcal{D} $. This probability $\mathbb{P}$ will play the role of an {\it a priori} probability. Given such $\mathbb{P}$ and an H\"older potential $V:S^1 \to \mathbb{R}$, we define a continuous time Ruelle operator, which is described by a family of linear operators $ \mathbb{L}^t_V$, $t\geq 0,$ acting on continuous functions $\varphi: S^1 \to \mathbb{R}$.
More precisely, given any H\"older $V$ and $t\geq 0$, the operator $ \mathbb{L}^t_V$, is defined by

\smallskip

$\,\,\,\,\,\,\,\,\,\,\,\,\,\,\,\,\,\,\varphi \,\to \psi(y) = \mathbb{L}^t_V(\varphi)(y)= \int_{w(t)=y} e^{ \int_0^t V(w(s)) \, ds} \, \varphi (w(0)) \,d \mathbb{P}(w).\,$

\smallskip

For some specific parameters we show the existence of an eigenvalue $\lambda_V$ and an associated H\"older eigenfunction
$\varphi_V>0$ for the semigroup $\mathbb{L}_V^t$, $t\geq 0.$ After a coboundary procedure we obtain another stochastic semigroup, with infinitesimal generator $L_V$, and this will define a new probability $\mathbb{P}_V$ on $\mathcal{D}$, which we call the Gibbs (or, equilibrium) probability for the potential $V$. In this case, we define entropy for some continuous time shift-invariant probabilities on $\mathcal{D}$, and we consider a variational problem of pressure.  Finally, we define entropy production and present our main result: we analyze its relation with time-reversal and symmetry of $L$. We also show that the continuous-time shift $\Theta_t$, acting on the  Skorokhod space $D$, is expanding.  We wonder if the point of view described here provides a sketch (as an alternative to the Anosov one) for the chaotic hypothesis for a particle
system held in a nonequilibrium stationary state, as delineated by Ruelle, Gallavotti, and Cohen.}

\end{abstract}

\vspace{0.1cm}
\vspace{.2cm}

\section{Introduction}

We consider the semi-flow given by the continuous time shift
$\Theta_t:\mathcal{D} \to \mathcal{D} $, $t \geq 0$,
acting on the
Skorokhod space $\mathcal{D} $ of \textit{c\`{a}dl\`{a}g} paths (right continuous with left limits) $w: [0,\infty) \to S^1$, where $S^1$ is the unitary circle (one can take $[0,1]$ instead of $S^1$). We will prefer to state the results in $[0,1]$. The set $\mathcal{D} $ is equipped with the Skorokhod metric. The Skorokhod space $\mathcal{D} $ is a noncompact Polish space.
We will show that continuous time shift
$\Theta_t$, $t \geq 0$, is expanding (see Proposition \ref{excam}).

Continuous time stochastic processes $X_t$, $t\geq 0,$ taking values on $[0,1]$ are described by probabilities $\mathbb{P}$ on $\mathcal{D} $. To say that the process is stationary is equivalent to saying that the associated probability $\mathbb{P}$ is invariant for the action of the shift $\Theta_t$, $t \geq 0$.

The results presented in the initial part of our work are in some way related to \cite{BEL}, \cite{LN} and \cite{LN1}. Our main purpose is to describe a version of Thermodynamic Formalism for semi-flows specified by infinitesimal generators. More precisely, in Section \ref{SRu1} we follow the program of introducing a Ruelle operator from a potential and an {\it a priori}
probability (in a similar fashion as in  \cite{BCLMS}, \cite{LMMS} and \cite{BEL}). 

We introduce a stochastic semi-group $e^{t\, L}$, $t \geq 0,$ where $L$ (the infinitesimal generator) acts on continuous functions $f:[0,1]\to\mathbb{R}$. This stochastic semigroup and an initial vector of probability $\pi$ define an associated stationary shift-invariant probability $\mathbb{P}$ on the $\mathcal{D} $ (see \cite{Li}). This probability $\mathbb{P}$ will play the role of an {\it a priori} probability (a continuous time version of the point of view of \cite{LMMS} and \cite{BCLMS}).

Given the {\it a prori} probability $\mathbb{P}$ on $\mathcal{D}$ and a H\"older continuous potential $V:[0,1] \to \mathbb{R}$, we define the Ruelle operator $\mathbb{L}^t_V$, $t \geq 0,$ in such way that for $\varphi :[0,1] \to \mathbb{R}$, we get
$\mathbb{L}^t_V(\varphi)=\psi$, $t\geq 0$, when
\begin{equation} \label{erry} \varphi \,\to \psi(y) = \mathbb{L}^t_V(\varphi)(y)= \int_{w(t)=y} e^{ \int_0^t V(w(s)) \, ds} \, \varphi (w(0)) \,d \mathbb{P}(w).\,
\end{equation}

The above expression can be recognized as in Feynman-Kac form if the infinitesimal generator is symmetric according to Figure \ref{fig1} (see also \cite{LN}).

The Feynmann-Kac formula is the partial differential equation
$$ \frac{\partial u}{\partial t}+ L u + V \, u =0.$$

General results for continuous-time Markov chains that were specially designed to be applicable to our setting appear on \cite{AG}.

Note that expression \eqref{erry} depends also on $L$ (because $\mathbb{P}$ depends on $L$).
 
See Section \ref{SRu1} for the discussion about existence of an eigenvalue and a positive eigenfunction for the continuous time  Ruelle operator. We are able to get solutions for some specific parameters.

In Theorem  \ref{theo3} we show that

\begin{theo}\label{theo31} For polynomials of degree two we show some explicit expressions for the main eigenvalue and eigenfunction (of the polynomial form) for the continuous time  Ruelle operator.
\end{theo}

Moreover, in Example \ref{tree}, we present explicit expressions for the eigenvalue and the eigenfunction $f:[0,1] \to \mathbb{R}$ solutions (of polynomial form) for a more general class of infinitesimal generators $L$ and potentials $V$ of polynomial form.  

We point out that when the state space is not  discrete (in our case is $[0,1]$), given a Lipschitz potential $V:[0,1] \to \mathbb{R}$, does not always exist an eigenfunction for the continuous time Ruelle operator 
(see discussion on Section \ref{SRu1}).  Example \ref{tree} displays the difficulties inherent in the search for polynomial expressions, but does not exclude the possibility of the existence of analytical expressions.

Assuming existence of the eigenfunction for $L$ and $V$, after a kind of coboundary procedure, we obtain another stochastic semigroup, with infinitesimal generator $L_V$, and this will define a new probability $\mathbb{P}_V$ on $\mathcal{D}$, which we call the equilibrium probability for the potential $V$ (see Definitions \ref{GiGi} and \ref{poi}, Lemma \ref{eeste} and expressions \eqref{gerador_gibs} and \eqref{semigrupo_gibs}).
The initial stationary vector of probability for such a process is given by Proposition
\ref{ppo}. Note that $V$ is completely independent of the dynamics of the shift $\Theta_t$, $t \geq 0$, and the {\it a priory} probability defined by $L$.

From the {\it a priori} probability $\mathbb{P}$ on $\mathcal{D}$, in Section \ref{SRE} we can introduce the concepts of entropy for a certain class of shift-invariant probabilities on $\mathcal{D}$, and pressure for a potential $V:[0,1] \to \mathbb{R}$ (see Definitions \ref{dede1} and \ref{press}).

More precisely, we define in Section \ref{SRE} the concept of Relative Entropy, we consider associated variational problem of Pressure and we define Equilibrium Probability for the potential $V$. We show in Theorem \ref{pois1} the exact expression \eqref{pois} for the Equilibrium probability for $V$.  This follows from results concerning the existence of a main eigenfunction for the Ruelle operator as described before.

In Section \ref{SEP} we define entropy production  and we discuss some properties related to time-reversal and we prove one of our main theorems: 

\begin{theo}\label{estamosal} The entropy production rate of the time reversal process is the same as the original process.
\end{theo}

Related results for continuous-time quantum channels (where the infinitesimal generator is a Lindbladian) appear in \cite{BKL}.

In Section \ref{SIK} we show 
\begin{theo}\label{excama} The continuous-time shift $\Theta_t$, acting on the  Skorokhod space $\mathcal{D}$, is expanding for the Skorokhod metric.
\end{theo}

In \cite{GaCo}, \cite{Ga}, \cite{Ru0}, \cite{LNP1} and \cite{Ru}, the authors use an idea of Ruelle's as a guiding
principle to describe nonequilibrium stationary states in general. The purpose is to better understand a model for the chaotic hypothesis for a single (moving) particle
system held in a nonequilibrium stationary state. This model is described by properties of SBR probabilities for Axiom A (or Anosov) systems and entropy production rate (see also \cite{Da1}, \cite{MNS}, \cite{MN} and \cite{LNP2}). In this case, the potential is fixed as the Lyapunov exponent. The reason for such interest is that the real physical problem behaves, in many respects,
as if they were Anosov systems as far as their properties of physical
interest are concerned.
We wonder if our setting, where $V$ is general, also provides a sketch (as an alternative for the Anosov one) for the chaotic hypothesis.

The Appendix Sections \ref{A1} and \ref{A2} are technical and aim to analyze some integral kernels that naturally appear in our reasoning.

Some of our results are related to the ones in \cite{DoVa}, \cite{Ki1}, \cite{Ki2}, \cite{MNS}, \cite{MN},  \cite{LN}, \cite{Da1}, \cite{Gomes}, \cite{LT} and \cite{Morales}.

\section{The Model} \label{int}

Consider an infinitesimal generator $L$ of a Markov jump process with jump rate function $\lambda\equiv 1$ and a kernel $P(x,dy)$ that can be decomposed as $P(x,y)dy$, where the continuous function $P:[0,1]^2\to [0,1]$ satisfies for all $y\in[0,1]$

\begin{equation} \label{kj} \int P(x,y) \ dx =1. \end{equation}

This operator acts on periodic functions $f:[0,1]\to\R$ by
\begin{equation}\label{L}(Lf)(x)= \int \big[ f(y) -f(x) \big] P(x,y)dy.\end{equation}
Notice that $L(1)=0.$ We call $L$ the {\it a priori} infinitesimal generator.

A trivial example is when $P(x,y)=1$, for all $(x,y).$

We will denote by $L^*$ the dual of $L$ in $\mathcal{L}^2(dx)$, which  acts on functions $g:[0,1] \to \mathbb{R}$ by
\begin{equation}\label{Ls} (L^*g)(x)= \int P(y,x)g(y)dy - g(x). \end{equation}

Let $\theta$ be the invariant vector for $P$ on the left. In the subsection entitled ``Markov Chains with values on $S^1$'' of \cite[Section 3]{LMMS}, it is shown that, under H\"{o}lder assumption, there exists a unique $\theta$. Define $\mu(dx)=\theta(x)dx$ the probability measure with density $\theta$. This means that it satisfies
\begin{equation}\label{equacaodotheta} \int \theta(y)P(y,x)dy = \theta(x). \end{equation}
By the above, we get $L^*(\theta)=0$, which means that $\mu$ is invariant for the action of $L^*$.

When  $P(x,y)=1$, for all $(x,y)$,  we get that $\theta=1$ is a solution

Notice that $L$ and $L^*$ are bounded operators. Then, we can define the semigroup $e^{t L}$. For fixed $t\geq 0$, this semigroup is an integral operator, that is, there exists a kernel function $K_{t}:[0,1]\times[0,1] \to \mathbb{R}^+$ such that
\begin{equation}\label{functK}(e^{tL}f)(x) = \int K_{t}(x,y)f(y)dy+e^{-t}f(x).\end{equation}
The existence of this function $K_t$ is presented in Appendix section \ref{A1} along with some properties that it satisfies. 

For a continuous-time Markov process $\{X_t, t\geq 0\}$, the kernel $K_t$ does not integrate to 1, but if we consider $\hat K_t(x,dy) = K_t(x,y)dy + e^{-t}\delta_x(dy)$, then $\int \hat K(x,dy) = 1$ for all $x \in [0,1]$. Thus, $\hat K$ plays the same role as the transition function on the discrete-time case. Given an initial density function $\varphi_0$, denote by $\mathbb{P}$ the probability induced by this process in $\mathcal{D}$; we can measure a cylinder set $\mathcal{C}=\{X_0\in (a_0,b_0) , X_{t_1}\in (a_1,b_1), X_{t_2}\in (a_2,b_2)\}$ by
$$\mathbb{P}(\mathcal{C}) = \int_{a_0}^{b_0} \int_{a_1}^{b_1} \int_{a_2}^{b_2} \hat K_{t_2- t_1}(x_1,dx_2 ) \hat K_{t_1}(x_0,dx_1) \, \, \varphi_0(x_0) dx_0.$$ 

Now, let us see how we can compute $K_t$ with an example:
\begin{exam} \label{eex1} \textit{Take $P(x,y)= \cos[ (x - y) 2 \pi ]/2 + 1 $. This $P$ is symmetric and continuous on $[0,1]$. Since $\int \cos[(x-y) 2 \pi] dy = 0$, for any $x\in[0,1]$, we get that $\int P(x,y) d y=1$. In this case, the kernel function $K_t(x,y),t \geq 0$ can be explicitly expressed by
$$K_t(x,y) = 2\cos[2\pi(x-y)] (e^{-3t/4}-e^{-t}) +(1-e^{-t})$$
and the Lebesgue probability $d x$ is the unique invariant probability.}

First, we will calculate the $K_t$ expression. Note that
$$ \int \cos(2\pi(x-z))\cdot  \cos(2\pi(z-y))\, dz = \frac{1}{2}\cos(2\pi(x-y)).$$
Using induction, we show that $P^{n}(x,y)=\frac{ \cos(2\pi(x-y))}{2^{2n-1}} +1:$
\begin{eqnarray*}
P^{n+1}(x,y)&:=&\int P^n(x,z)P(z,y) dz\\
&=& \int \left(\frac{ \cos[2\pi(x-z)]}{2^{2n-1}} +1 \right)\left(\frac{ \cos[2\pi(z-y)]}{2} + 1\right)  dz  \\
&=& \frac{1}{2^{2n}} \int  \cos[2\pi(x-z)]\cdot  \cos[2\pi(z-y)]\, dz + 1\\
&=& \frac{ \cos[2\pi(x-y)]}{2^{2n+1}} + 1.
\end{eqnarray*}

By the general case, see Appendix section \ref{A1}, we know that
$$K_t(x,y)= \sum_{k=1}^{\infty} \frac{t^k}{k!} Q_k(x,y),$$
where
\begin{eqnarray*}
Q_k(x,y) &:=& \sum_{j=1}^{k} (-1)^{k-j}{k \choose j} P^j(x,y)\\
&=& \sum_{j=1}^{k} (-1)^{k-j}{k \choose j} \left(\frac{ \cos[2\pi(x-y)]}{2^{2j-1}} +1\right)\\
&=& 2\cos[2\pi(x-y)] \sum_{j=1}^{k} (-1)^{k-j}{k \choose j} \frac{1}{2^{2j}} +\sum_{j=1}^{k} (-1)^{k-j}{k \choose j}\\
&=& 2\cos[2\pi(x-y)] \left[(-1)^{k+1}+\left(-\frac{3}{4}\right)^k\right] +(-1)^{k+1}.
\end{eqnarray*}
The above gives us exactly the formula we want for $K_t(x,y)$.

Now, we turn ourselves to the second claim. The fact that $dx$ is invariant is an immediate consequence of symmetry: the function $1$ satisfies $L^*(1)=L(1)=0$. We need to go further to get uniqueness. 

A continuous function $f:[0,1] \to \mathbb{R}$ can be seen as a periodic function $f:\mathbb{R} \to \mathbb{R}$ with period $1$ so that we can employ Fourier Series. Write
$$ f(x)=\dfrac{a_0}{2} + \sum_{n=1}^{\infty} a_n \cos(2\pi n x) + \sum_{n=1}^{\infty} b_n \sin(2\pi n x),$$
with $\frac{a_0}{2}=\int f(x)dx$, $a_n=2\int f(x)\cos(2 \pi n x)dx \, $ and  $\, b_n=2\int f(x)\sin(2 \pi nx)dx$. Notice that $\cos(2\pi (x-y))=\cos(2\pi x)\cos(2\pi y) + \sin(2\pi x)\sin(2\pi y).$ Then
\begin{eqnarray*}
(Lf)(x) &=& \int f(y)dy + \frac{1}{2} \int f(y)\cos[2\pi (x-y)]dy - f(x)\\
&=& \dfrac{a_0}{2} +\frac{1}{2} \cos(2\pi x) \frac{a_1}{2} + \frac{1}{2} \sin(2\pi x) \frac{b_1}{2} -f(x).
\end{eqnarray*}

Therefore, $Lf=0$ if, and only if,
$$f(y) = \dfrac{a_0}{2} + \cos(2\pi y) \frac{a_1}{4} + \sin(2\pi y) \frac{b_1}{4}$$
and consequently $a_1=a_1/4$, $b_1=b_1/4$ and $a_n=b_n=0$, $\forall n \ge 2$. We conclude that $L^*f=Lf=0 \Leftrightarrow f \equiv \frac{a_0}{2}$, constant. It means that the only eigendensity of the operator $e^{tL}$ is that of Lebesgue measure $dx\, \, (\theta \equiv 1)$. 
\end{exam}

Consider $L$ as defined in equation \eqref{L}, let us assume that there exists a positive continuous density function $\theta: [0,1] \to \mathbb{R}$, such that, for any continuous function
$f:[0,1] \to \mathbb{R}$, we get
\begin{equation} \label{ktr} \int (Lf)(x) \theta(x)dx =0. \end{equation}
Moreover, as a consequence of the relation above, valid for any $f$, it is easy to see that $\theta$ is also a solution of equation \eqref{equacaodotheta}, which is unique under the H\"older assumption. Therefore, we can assume $L$ is such that the above-defined $\theta$ is unique.

\begin{defi} \label{kas}
Given $L$ defined on equation \eqref{L} and an initial density $\theta$ satisfying equation \eqref{ktr}, we get a continuous-time stationary Markov process $\{X_t, t \geq 0\}$, with values on $[0,1]$ (see \cite{BGL,Bo,Li}). This process defines a probability $\mathbb{P}$ on the  Skorokhod space $\mathcal{D}$. This probability $\mathbb{P}$ is invariant for the shift $\{\Theta_t$, $t \geq 0\}$, which acts on $\omega\in \mathcal{D}$ as $(\Theta_t (w))_ s=w_{s+t}$
\end{defi}

For this infinitesimal generator, the associated semigroup satisfies $e^{t\, L}(1)=1.$ Moreover, $e^{t\, L^*}(\theta)=\theta,$ where $L^*$ was given by equation \eqref{Ls} and $\theta$ satisfies equation \eqref{ktr}.

Now, we consider the continuous potential $V:[0,1] \to \mathbb{R}$ and introduce the operator $L+V$. In the same way, we define $L^*+V$. If $P(x,y)$ is symmetric, their spectral properties are the same. We also have a formula for the homogeneous semigroup
\begin{equation}\label{expLV}\left(e^{ t \, (L + V)}f\right)(x)=\mathbb E_x\left[e^{\int_0^tV(X_r)dr}f(X_t)\right],\end{equation}
where $\{X_t$, $t \geq 0\}$, is the Markov process with infinitesimal generator $L$. Notice that this semigroup is not Markovian.

Similarly to $e^{tL}$, this semigroup is also an integral operator, that is, there exists a kernel function $K_t^V:[0,1]\times[0,1] \to \mathbb{R}^{+}$ such that
\begin{equation}\label{functKV}\left(e^{ t \, (L + V)}f\right) (x) = \int K_t^V(x,y)f(y) dy + e^{-t} e^{tV(x)}f(x).\end{equation}
The properties of $K_t^V$ are the ones presented in Appendix Section \ref{A2} if we consider $\lambda\equiv 1$.


\section{Ruelle Operator and the Gibbs Markov process} \label{SRu1} 

In this section, we will introduce the Ruelle operator (which was considered in similar cases in \cite{BEL}) and use a kind of normalization procedure to get the Gibbs Markov process and its induced Gibbs probability on $\mathcal{D}$. 

\begin{defi}[Ruelle Operator] \label{Ruru} Consider, for a fixed $t\geq 0$, the continuous-time Ruelle operator $\mathbb{L}^t_V$, associated with $V$, that acts on continuous functions $\varphi: [0,1]\to\R$ as
$$ (\mathbb{L}^t_V \varphi)(x)= \E_x\left[e^{\int_0^t V(X_r)dr}\varphi(X_t)\right]=\left(e^{t(L+V)}\phi\right)(x).$$
\end{defi}

For a symmetrical $L$, as the Feynman-Kac formulas for natural and reverse time processes coincide, this Ruelle operator can be seen as the continuous-time version of the classical Ruelle operator (discrete case). Figure \ref{fig1} depicts this statement. The left approach is more suitable for the Feymann-Kac formula, while the right one can be easily related to the classical discrete-time Ruelle operator for the $n$-coordinate shift $\sigma^n:[0,1]^\N\to [0,1]^\N$, with $y$ being the initial value of the shifted path. 
\begin{figure}[h]
	\centering
	\hspace{-10pt} {\includegraphics[scale=0.8]{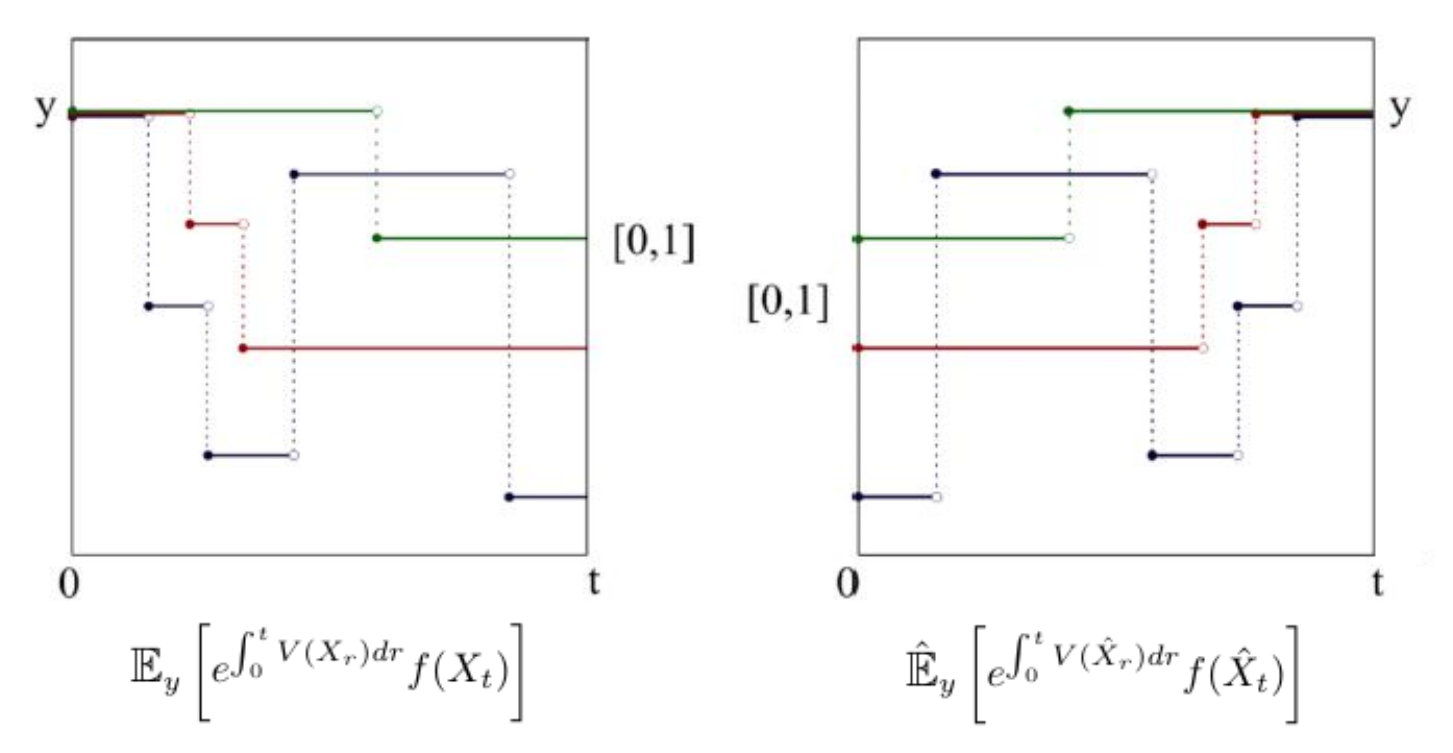}}
\caption{If $L$ is symmetric, we can use the Ruelle Operator at natural or reversal time.}
\label{fig1}
\end{figure} 

According to our notation, the continuous-time Ruelle operators $\mathbb{L}^t_V$, $t \geq 0$, are a family of linear operators indexed by $t$.
\begin{defi} \label{Rururu} Fix $V:[0,1] \to \mathbb{R}$. We say that the family of Ruelle operators $\mathbb{L}^t_V, t \geq 0,$ is normalized if $\mathbb{L}^t_V 1=1$, for all $t\geq 0.$
\end{defi}

If the potential $V\equiv 0$, for any $t\geq 0$, the Ruelle operator is $\mathbb{L}_0^t = e^{tL}$. In this case, the family of Ruelle operators is normalized. From now to the end of this section, we will study non-normalized Ruelle operators to associate them with a normalized Gibbs operator.

\begin{defi} \label{Rural} We say that $f:[0,1] \to \mathbb{R}$ is an eigenfunction on the right of the Ruelle operator $\mathbb{L}^t_V$, $t \geq 0$,
associated with the eigenvalue $\lambda\in \mathbb{R}$, if for all $t\geq 0$,
$$ \mathbb{L}^t_V f = e^{\lambda\, t} f.$$
Similarly, we say that $h$ is an eigenfunction on the left, meaning an eigenfunction of $(\mathbb{L}^t_V)^*: L^2([0,1],dx) \to  L^2([0,1],dx)$, if 
$$h\,\mathbb{L}^t_V := (\mathbb{L}^t_V)^*\, h = e^{\lambda\, t} h.$$
\end{defi}

To find these eigenfunctions, we have to analyze the properties of the operator $L+V$ and $L^* + V.$

Assume that $f,h$ are positive functions such that
\begin{equation} \label{rtyo} (L+V) (f)=\lambda f\hspace{5pt}\textrm{ and }\hspace{5pt}(L^{*}+V) (h)=\lambda h,\end{equation}
that is, $f,h:[0,1] \to \mathbb{R}^+$ are eigenfunctions of $L+V$ and $L^{*}+V$, respectively, associated with the same eigenvalue $\lambda\in \mathbb{R}$. Then, $e^{t(L + V)}f = e^{\lambda t} f,$ which makes $f$ an eigenfunction for the associated Ruelle operator. In addition, $e^{t(L^* + V)}h = e^{\lambda t} h$. We say that such $\lambda$ (positive or negative) is the main eigenvalue.

Notice that, by linearity, we have a whole class of functions that satisfies equation \eqref{rtyo}. It is natural to assume the normalization condition $\int h(x) dx =1$, so we can see $h$ as a density. Let us take the specific $f$ that satisfies
$\int f (x) h (x) d x =1.$
In this case, $\pi(x) = f (x) h (x) $ is a density on $[0,1]$. There is no eigenfunction for \eqref{rtyo}  when $ P=1$ and $V$ is not constant. 

In the following, we will assume a solution exists for equation \eqref{rtyo}. Comparing with \cite[pages 106 and 111]{Str}, we can see that, when the state space is discrete, we always  get that solutions via the Perron-Frobenius theory.

\begin{hyp}\label{assumPF} Assume that there exists an eigenvalue $\lambda\in \mathbb{R}$ and two functions $\ell:[0,1] \to \mathbb{R}^+$ and $r:[0,1] \to \mathbb{R}^+$ of H\"older class, such that,
\begin{eqnarray*}
(L+V) r=\lambda r&\textrm{ and }&\ell \,(L+V) =\lambda \ell.
\end{eqnarray*}
\end{hyp}

There are plenty of examples of pairs $L$ and $V$ for such ones the above condition is satisfied. To exemplify that, we will use a continuous function $g:[0,1]\to\R^{+}$, satisfying $\int g(x)dx=1$, to define $P(x,y)$, for $x,y\in[0,1]$, by
$$P(x,y)=\left\{\begin{array}{ll}g(x+y),& \textrm{ if } (x+y)<1,\\ g(x+y-1),& \textrm{ if } (x+y)\geq 1.\end{array}\right.$$
This $P$ is a symmetric kernel and the corresponding invariant density $\theta$, satisfying equation \eqref{equacaodotheta}, is equal to 1. Therefore, the $g$ we choose defines $L$ via $P$. Notice that $L^{*}=L$, then we just need to find $f$ such that $(L+V)(f)=\lambda f$, because, in this case, we have $(L^{*}+V)h=\lambda h$ for $h=f$.

\begin{theo}\label{theo3} Consider $g$ as a restriction of a polynomial of degree 2 to $[0,1]$, with $g(0)=g(1)\geq 0$. Assume $L$ is defined via $P$ using the polynomial $g$, as discussed above. Defining, for any $b\in\R$, a polynomial $V(x)=b+[1-g(0)]x(1-x)$, there exists $\lambda\in\R$ and $f$ a polynomial of degree 2, with $f(0)=f(1)$, satisfying
\begin{equation}\label{exPF}
(L+V)(f)(x)=\int [f(y)-f(x)]P(x,y)dy + V(x)f(x) = \lambda f(x).
\end{equation}
Moreover, there exists a solution $f$ which is positive on $[0,1]$.
\end{theo}

\begin{proof}
Write $g(x)=a_0+a_1x+a_2x^2$, for some $a_0,a_1,a_2\in\R$ with $a_0>0$ and $a_2\neq 0$. The restriction $g(0)=g(1)$ implies that $a_2=-a_1$, while the integral condition gives us that $a_1=6(1-a_0)$. Considering both, we have $g(x)=a_0+6(1-a_0)x(1-x)$ with $a_0>0$ and $a_0\neq 1$. 

In the same way, the polynomial $f$ is of the form $f(x)=c_0+c_1x(1-x)$, for some $c_0,c_1\in\R$. Using the definition of $P$, we can compute the integral term of equation \eqref{exPF} as
\begin{eqnarray*}p(x)&:=&\int_0^{1-x} [f(y)-f(x)]g(x+y)dy + \int_{1-x}^1 [f(y)-f(x)]g(x+y-1)dy
\\&=&\frac{1}{30}(6-a_0)c_1-c_1x+a_0c_1x^2+2(1-a_0)c_1x^3-(1-a_0)c_1x^4.
\end{eqnarray*}

Considering $g(0)=a_0$ on the definition of $V$, the expression $p(x)+V(x)f(x)$ of the left hand side of equation \eqref{exPF} turns into
$$\left[bc_0+\frac{c_1}{5}-\frac{a_0c_1}{30}\right]+\left[(1-a_0)c_0-(1-b)c_1\right]x+\left[(-1+a_0)c_0+(1-b)c_1\right]x^2.$$
This expression needs to be equal to $\lambda f(x)$, also a polynomial of degree 2, for some $\lambda\in\R$. This means that these polynomials should have the same coefficients, which gives us three equations:
$$\left\{\begin{array}{l}
bc_0+\frac{c_1}{5}-\frac{a_0c_1}{30} = \lambda c_0, \\
(1-a_0)c_0-(1-b)c_1 = \lambda c_1, \\
(-1+a_0)c_0+(1-b)c_1 = - \lambda c_1.
\end{array}\right.$$

First, notice that the last two equations give us the same condition. As $a_0\neq 1$, they give us that $c_0=\left(\frac{b-1-\lambda}{a_0-1}\right)c_1.$ Substituting on the first one, we get
$$c_1\left(\left[\frac{30b^2-30b-6+7a_0-a_0^2}{30(a_0-1)}\right]-\left[\frac{2b-1}{a_0-1}\right]\lambda+\left[\frac{1}{a_0-1}\right]\lambda^2\right)=0.$$
The general solution of this is $\lambda=\frac{1}{30}\left(-15\pm\sqrt{405-210 a_0 + 30a_0^2}+30b\right),$ which gives
\begin{equation}\label{second}f(x)=c_1\left(\frac{b-1-\lambda}{a_0-1}+x(1-x)\right).\end{equation}

The general eigenfunction $f$ presented above can, sometimes, assume negative values on the interval $[0,1]$. The Assumption \ref{assumPF} asks positivity since we need that in our reasoning (see equation \eqref{gamma}), but this is not a problem in this case because we have positivity for $0<a_0<1$ if we take $c_1>0$ and for $a_0\geq 1$ if $c_1<0$. 
\end{proof}

\begin{remm} The same may not be true in other cases. For instance, considering $g$ and $V$ as polynomials of degree 3, we were able to get a positive polynomial $f:[0,1]\to\R^{+}$ of degree 3, but the next example shows that it is impossible to get a polynomial solution on $S^1$, considering the periodic boundary condition $0\equiv 1$. In the search for suitable $V$ and $f$, we are not able to get solutions satisfying the property that $g$ is strictly positive. This is not very intuitive since, for larger degrees, we have more free variables to work. In fact, in the degree $n$ case, we have $3n+4$ coefficients ($n+1$ from each polynomial and one from $\lambda$) to cancel $2n+1$ coefficients of $p(x)+(V(x)-\lambda)f(x)$, the constraints, which means there are left three coefficients for periodicity, one for $\int g(x)dx=1$ and at least $n-1$ to adjust positivity.
\end{remm}

\begin{exam}  \label{tree} Consider periodic polynomials $f,g,V:[0,1]\to\R$ of degree 3 as
$$g(x)=a_0+a_1x-3(-4+4a_0+a_1)x^2+(-a_1+3(-4+4a_0+a_1))x^3,$$
$$V(x)=b_0+b_1x+b_2x^2+(-b_1-b_2)x^3,$$
$$f(x)=c_0+c_1x+c_2x^2+(-c_1-c_2)x^3$$
and define a polynomial of degree 6 by
$$K(x):=\int_0^1 P(x,y)f(y)dy+V(x)f(x)-(1-\lambda)f(x).$$

Suppose that $c_1,c_2\in\R$ are such that:
\begin{enumerate}[1)]
\item $c_1\neq 0$;
\item $c_1+c_2\neq 0$;
\item $2c_1+c_2\neq 0$; 
\item $3c_1+c_2\neq 0$; 
\item $54c_1^2+39c_1c_2+8c_2^2\neq 0$;
\item $91368 c_1^6 + 186948c_1^5c_2 + 159318c_1^4c_2^2 + 71367c_1^3c_2^3 + 17228c_1^2c_2^4 + 1984c_1c_2^5 + 64c_2^6\neq 0$.
\end{enumerate}
In this case, solving $K(x)=0$, we find expressions\footnote{We are not showing the exact expressions here due to its complexity, but it is possible to get them using the \textit{Mathematica} software by nullifying each coefficient of $K$ from the highest exponent to the smallest one.} for $b_2,b_1,a_1,c_0,\lambda,a_0$ as functions of $b_0,c_1,c_2$. Thus, we get a whole class of polynomials $f,g,V$ that solves $K(x)=0$, for all $x$, but none of these solutions satisfies that $g$ is strictly positive on $[0,1]$. In order to see that, let us analyze the properties of a generic polynomial $G:\R\to\R$ of degree 3 satisfying $G(0)=G(1)$ and $\int_0^1 G(x)dx=1$.

If we write $G$ as
$$G(x)=\frac{1}{12}(12+2g_2+3g_3)+(-g_2-g_3)x+g_2x^2+g_3x^3,$$
the expressions for the two critical points of $G$ are
$$X_1=\frac{g_2-\sqrt{g_2^2+3g_2g_3+3g_3^2}}{3g_3}\hspace{5pt}\textrm{ and }\hspace{5pt}X_2=\frac{g_2+\sqrt{g_2^2+3g_2g_3+3g_3^2}}{3g_3}.$$
Looking closely at these critical points, we see that $X_1$ is a local maximum and $X_2$ is a local minimum for $G$. Furthermore, since $G(0)=G(1)$, there is at least one critical point of $G$ on $[0,1]$. Then, we can divide the analysis into three cases:
\begin{enumerate}[(i)]
\item $X_1\in[0,1]$ and $X_2\not\in[0,1]$;
\item $X_2\in[0,1]$ and $X_1\not\in[0,1]$;
\item $X_1,X_2\in[0,1]$.
\end{enumerate}
Visually, we can see these cases as in Figure \ref{G-graficos}

\begin{figure}[h] 
\center
\includegraphics[width=\textwidth]{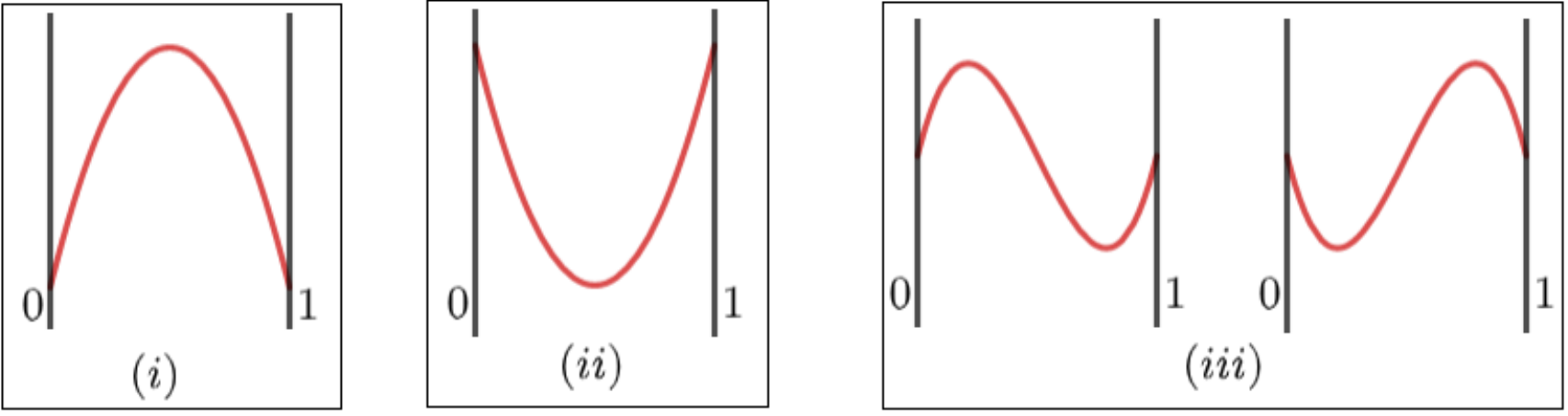}
\caption{Graphics}
\label{G-graficos}
\end{figure}

\noindent Notice that these conditions are uniquely defined by the sign of the derivatives of $G$ on $x=0$ and $x=1$. Moreover, the positivity is given by the absolute minimum on $[0,1]$, which is $G(0)=G(1)$ on \textit{(i)} and $G(X_2)$ on \textit{(ii)} and \textit{(iii)}. 

Finally, we consider $g_2$ and $g_3$ as $a_2$ and $a_3$, the respective coefficients of $g$ that solves $K(0)=0$. Using the free variables $c_1,c_2\in\R$, in all three cases, we can reduce $g(x)\geq 0$, for all $x\in[0,1]$, into 
$$\left\{\begin{array}{ll}
c_2\neq 0,&\textrm{if $c_1=0$},\\
c_2=-2c_1,&\textrm{if $c_1\neq 0$}.
\end{array}\right.$$
In both cases, there is no solution, since we have a clash with conditions 1 and 3 that we initially set for $c_1,c_2$.
\end{exam}

Observe that a function $f$ obtained from Assumption \ref{assumPF} defines a whole set of functions $\{\alpha f, \alpha\in\R^{+}\}$ that also satisfies the same condition. In Example \ref{tree}, this is very clear when we look at equation \eqref{second}. One can use this subspace to get specific functions satisfying some conditions. For instance, for a fixed $V$, we take $\ell_V$, $r_V$ and $\lambda_V$ the ones from Assumption \ref{assumPF} that satisfies the normalization conditions
\begin{equation}\label{rrty}
\int \ell_V (x) dx =1\hspace{5pt}\textrm{ and }\hspace{5pt}\int r_V(x) \ell_V(x) dx=1.
\end{equation}
An equation for the right eigenfunction $r_V$ is
\begin{equation} \label{pot}
\int P(x,z) r_V(z) dz  - (1 + \lambda_V  - V(x) )r_V(x) =0,
\end{equation}
for any $x$. On the other hand, the left eigenfunction $\ell_V$ satisfies
\begin{equation} \label{pot2}
\int \ell_V(z)   P(z,x)  dz  - (1 + \lambda_V - V(x) ) \ell_V(x) =0.
\end{equation}
For all $x,y\in [0,1]$, $t\geq 0$ and $f\in C_b([0,1])$, define
\begin{equation}\label{gamma}
\gamma_V(x)=1+\lambda_V -V(x),\qquad Q_V(x,y)=\frac{P(x,y)r_V(y)}{r_V(x)\gamma_V(x)},
\end{equation}
\begin{equation}\label{gerador_gibs}(\mathcal{L}_Vf)(x)=\gamma_V(x)\int[f(y)-f(x)]Q_V(x,y)\,dy\end{equation}
and
\begin{equation}\label{semigrupo_gibs}(\mathcal{P}^V_tf)(x)=\frac{e^{t(L+V)}(r_V f)(x)}{e^{\lambda_V \,t}r_V(x)}.\end{equation}

\begin{remm} One can also write
$$(\mathcal{P}^V_t f) (x) = \frac{1}{e^{\lambda_V t} r_V( x)} \mathbb{E}_x\left[ e^{ \int_0^t V( X_s )\, ds} r_V(X_t) f(X_t)\right].$$
\end{remm}

\begin{lemma} \label{eeste}
The operator $\mathcal{P}^V_t$ is the semigroup associated with the infinitesimal generator $\mathcal{L}_V$, that is,
$$\lim_{t\to 0}\frac{(\mathcal{P}^V_t f)(x)-f(x)}{t}=(\mathcal{L}_V f)(x).$$
\end{lemma}

\begin{proof}
We can rewrite
\begin{eqnarray*}
\frac{(\mathcal{P}^V_t f)(x)-f(x)}{t} &=& \frac{1}{e^{\lambda_V t}r_V(x)}\left(\frac{e^{t(L+V)}(r_V f)(x)-r_V(x)f(x)}{t}\right)
\\&&+f(x)\left(\frac{e^{-\lambda_V \,t}-1}{t}\right).
\end{eqnarray*}

Taking limit as $t\to 0$, we get
\begin{eqnarray*}
&&\frac{1}{r_V(x)}(L+V)(r_V f)(x)-\lambda_Vf(x)\\*
&=&\frac{1}{r_V(x)}\left[\int P(x,y)r_V(y)f(y)dy+(V(x)-1)r_V(x)f(x)\right]-\lambda_Vf(x)\\
&=&\frac{1}{r_V(x)}\int P(x,y)r_V(y)f(y)dy-\gamma_V(x)f(x)\\
&=&\gamma_V(x)\int \left(\frac{P(x,y)r_V(y)}{r_V(x)\gamma_V(x)}\right)f(y)dy-\gamma_V(x)f(x).
\end{eqnarray*}

Using equation \eqref{pot}, we have
\begin{equation}\label{orp}
\int Q_V(x,y) dy =\int \frac{P(x,y)r_V(y)}{r_V(x)\gamma_V(x)} dy = 1.
\end{equation}

Then,
$$\lim_{t\to 0}\frac{(\mathcal{P}^V_t f)(x)-f(x)}{t}=\gamma_V(x)\int [f(y)-f(x)]Q_V(x,y)dy=(\mathcal{L}_V f)(x).$$
\end{proof}

Notice that the semigroup $\mathcal{P}^V_t$, $t \geq 0,$  is normalized. Furthermore, equation \eqref{orp} defines a jump process (with the kernel given by $Q_V$ and jump rate function $\gamma_V$).

\begin{defi}[Gibbs Markov process]\label{GiGi} We call Gibbs Markov process associated with the potential $V$ (and the a priori infinitesimal generator $L$) the continuous-time Markov jump process generated by $\mathcal{L}_V$.
\end{defi}

Now, we want to prove that the invariant density for the Gibbs process is $\pi_V=\ell_V r_V$, where the normalization conditions given by equation \eqref{rrty} are assumed to be satisfied. We need to use the dual operator $\mathcal{L}^*_V$ to do this.

\begin{lemma}\label{lemdual}
The dual of the operator $\mathcal{L}_V$ in $L^2([0,1],dx)$ is the operator
\begin{equation}\label{trw}
(\mathcal{L}_V^*g)(x)= \int \gamma_V(y) g(y) Q_V(y,x) dy - \gamma_V(x) g(x).
\end{equation}
\end{lemma}

\begin{proof} Given the functions $f,g$ we get
\begin{eqnarray*}
&&\int (\mathcal{L}_V f)(x) g(x) dx\\*
&=&\int \gamma_V(x)g(x)\int [f(y)-f(x)]Q_V(x,y) dy dx\\
&=&\int \int \gamma_V(x)g(x) Q_V(x,y)f(y)dy dx  -\int \gamma_V(x)f(x) g(x) \left[\int Q_V(x,y) dy \right] dx\\
&=&\int f(y) \int \gamma_V(x)g(x) Q_V(x,y) dx dy  -\int \gamma_V(x)f(x) g(x) dx\\
&=&\int f(z)\left[\int \gamma_V(x)g(x) Q_V(x,z) dx  - \gamma_V(z)g(z)\right] dz = \int f(z) (\mathcal{L}_V^*g)(z) dz.
\end{eqnarray*}
\end{proof}

Next, we show that $\pi_V$ is the stationary density.

\begin{prop} \label{ppo} 
The density $\pi_V$ satisfies $\mathcal{L}_V^*(\pi_V)=0.$
\end{prop}

\begin{proof} From equations \eqref{pot2} and \eqref{gamma} we get that, for any point $x$,
\begin{eqnarray*}
(\mathcal{L}_V^*\pi_V)(x)&=&\int \gamma_V (y) \ell_V(y) r_V(y) Q_V(y,x) dy - \gamma_V(x) \ell_V(x) r_V(x)\\
&=& \int \gamma_V(y) \ell_V(y) r_V (y)  \left(\frac{P(y,x)r_V(x)}{r_V(y)\gamma_V(y)}\right) dy - \gamma_V(x)\ell_V(x) r_V (x)\\
&=& r_V(x) \int \ell_V(y) P(y,x) dy - \gamma_V(x)\ell_V(x) r_V (x)\\
&=& r_V(x)\left[\int \ell_V(y) P(y,x) dy - \gamma_V(x)\ell_V(x)\right]=0.
\end{eqnarray*}
\end{proof}

From the above, we get that for any $x$,
$$\int \gamma_V(y)\pi(y)Q_V(y,x)dy = \gamma_V(x) \pi_V(x).$$

\begin{remm} We have
$$(\mathcal{P}^V_t 1)  = e^{t\,\mathcal{L}_V } (1)=1$$
and
$$(\mathcal{P}^V_t)^*(\pi_V) =   (e^{t\,\mathcal{L}_V^* }  \pi_V) = \pi_V.$$
\end{remm}

\begin{defi}[Gibbs probability]\label{poi} The probability $\mathbb{P}^V$ induced on $\mathcal{D}$ by the Gibbs Markov process (with infinitesimal generator $\mathcal L_V$ and stationary probability $\pi_V$) will be called the Gibbs probability for the potential $V$ (and the a priori infinitesimal generator $L$). This $\mathbb{P}^V$ is  invariant   for the      shift $\{\Theta_s, s \geq 0\}$.
\end{defi}

In the case $V\equiv 0$, $\mathbb{P}^V$ is the a priori probability $\mathbb{P}$ of Definition \ref{kas}.


\section{Relative Entropy,  Pressure and the equilibrium state for $V$} \label{SRE}

In this section, we will consider a variational problem in the continuous-time setting, analogous to the pressure problem in the discrete-time setting. This requires a meaning for entropy. Thus, we will define the relative entropy. A continuous-time stationary Markov process that maximizes our variational problem is called {\it continuous-time equilibrium state for $V$}. 

Consider the infinitesimal generator $\tilde{\mathcal L}$, which acts on bounded measurable functions $f:[0,1]\to \mathbb R$ as
$$(\tilde{\mathcal L} f)(x)= \int \big[f(y)-f(x)\big]\tfrac{\phi(y)}{\phi(x)}P(x,y)dy,$$
where $\phi\in C_b([0,1])$. To rewrite the operator above on the traditional form of an infinitesimal generator, we consider
$$(\tilde{\mathcal L} f)(x)=\tilde{\gamma}(x)\, \int \big[f(y)-f(x)\big]\tilde Q(x,y)\,dy\,.$$

\begin{theo} \label{pois1}
The invariant probability for $\tilde{\mathcal L}$ is
\begin{equation} \label{pois} \tilde\mu(dx)=\frac{\phi(x)\tilde \ell_\phi(x)}{\Vert\phi\Vert\Vert \tilde{\ell}_\phi\Vert}dx,\end{equation} 
where  $\tilde \ell_\phi$ satisfies
$$\frac{1}{\tilde \ell_\phi(x)}\int \tilde\ell_\phi(y)P(y,x) dy=\tilde \gamma (x).$$
\end{theo}

\begin{proof}
Repeating the computation we did on the proof of Lemma \ref{lemdual}, we can show that, for any density $g$,
\begin{eqnarray*}
(\tilde{\mathcal{L}}^*g)(x)&=& \int \tilde{\gamma}(y) g(y) \tilde Q(y,x) dy - \tilde{\gamma}(x) g(x)\\
&=& \int g(y) \frac{\phi(x)}{\phi(y)}P(y,x) dy - \tilde{\gamma}(x) g(x).
\end{eqnarray*}

In particular, for $g=\frac{\phi\tilde{\ell}_\phi}{\Vert\phi\Vert\Vert\tilde{\ell}_\phi\Vert}$, we have
$$(\tilde{\mathcal{L}}^*g)(x)= \frac{\phi(x)}{\Vert\phi\Vert\Vert\tilde{\ell}_\phi\Vert}\left( \int \tilde{\ell}_\phi(y)P(y,x)dy- \tilde{\gamma}(x) \tilde{\ell}_\phi(x)\right)=0.$$
\end{proof}

\begin{defi}The probability $\tilde{\mathbb P}_{\tilde\mu}$ on D is called admissible if it is induced by the continuous-time Markov chain with infinitesimal generator $\tilde{\mathcal L}$ and initial measure $\tilde \mu$. 
\end{defi}

For $\tilde{\mathbb P}_{\tilde \mu}$ admissible and  $\mathbb P_{\tilde \mu}$ the probability induced by the original continuous-time Markov chain  with infinitesimal generator $ L$, defined in equation \eqref{L}, and initial probability $\tilde \mu$, define for $T>0$,
$$H_T(\tilde{\mathbb P}_{\tilde \mu}\vert\mathbb P_{\tilde \mu})=-\int_D \log\left(\left.\frac{d\tilde{\mathbb P}_{\tilde \mu}}{d\mathbb P_{\tilde \mu}}\right|_{\mathcal F_T}\right)(\omega)d\tilde{\mathbb P}_{\tilde \mu}(\omega).$$
Notice that we are using the same initial measure $\tilde{\mu}$ for both processes, so the probabilities are absolutely continuous with respect to each other.

Using this $H_T$ above defined, we introduce a meaning for the relative entropy.

\begin{defi}[Relative entropy]\label{dede1} For a fixed initial probability $\tilde{\mu}$, the limit
$$H(\tilde{\mathbb P}_{\tilde \mu}\vert\mathbb P_{\tilde \mu})=\lim_{T\to\infty}\dfrac{1}{T}H_T(\tilde{\mathbb P}_{\tilde \mu}\vert\mathbb P_{\tilde \mu})$$
is called the relative entropy of $\tilde{\p}_{\tilde{\mu}}$ with respect to $\p_{\tilde{\mu}}$.
\end{defi}

Since $L$ and $\tilde{\mathcal{L}}$ both generate Markov jump processes, we have a Radon-Nikodym derivative of them (see \cite{AG}). This implies that
\begin{eqnarray*}
\log\left(\left.\frac{d\tilde{\mathbb P}_{\tilde \mu}}{d\mathbb P_{\tilde \mu}}\right|_{\mathcal F_T}\right)(\omega)&=&\int_0^T [1-\tilde{\gamma}(\omega_s)]ds+\sum_{s\leq T}\log\left(\tilde{\gamma}(\omega_{s-})\frac{\varphi(\omega_s)}{\varphi(\omega_{s-})\tilde{\gamma}(\omega_{s-})}\right)\\
&=&\int_0^T [1-\tilde{\gamma}(\omega_s)]ds+\sum_{s\leq T}\left\{\log(\varphi(\omega_s))-\log(\varphi(\omega_{s-}))\right\}\\
&=&\int_0^T [1-\tilde{\gamma}(\omega_s)]ds+\log\left(\phi(\omega_T)\right)-\log\left(\phi(\omega_{0})\right).
\end{eqnarray*}
Then,
\begin{equation}\label{HH}
H(\tilde{\mathbb P}_{\tilde \mu}\vert\mathbb P_{\tilde \mu})=\int [\tilde{\gamma}(x)-1]d\tilde\mu(x).
\end{equation}

For a H\"older class potential $V$, the probability $\p^V_{\pi_V}$ is admissible. Then,
\begin{equation}\label{GibbsPre}H(\p^V_{\pi_V}\vert\p_{\pi_V})=\int [\gamma_V(x)-1]d\pi_V(x)=\lambda_V-\int V(x)d\mu_V(x).\end{equation}

\begin{defi}[Pressure]\label{press} We denote the Pressure (or Free Energy) of $V$ as the value
$$\textbf{P}(V):= \sup_{\at{\tilde{\mathbb P}_{\tilde \mu}}{\text{ admissible}}}\left\{H(\tilde{\mathbb P}_{\tilde \mu}\vert\mathbb P_{\tilde \mu})+\int V(x) d\tilde \mu(x)\right\}.$$
\end{defi}

Using equation \eqref{HH}, the pressure can be written as
$$\textbf{P}(V)= \sup_{\at{ \tilde{\mathbb P}_{\tilde \mu}}{\text{admissible}}}\int [\tilde{\gamma}(x)-1+V(x)]\,d\tilde\mu(x).$$
Recalling the expressions of $\tilde \gamma$ and $\tilde \mu$, we have
$$\textbf{P}(V)=\sup_{\phi>0} \int \left(\frac{\tilde \ell_\phi}{\Vert\tilde \ell_\phi \Vert}\right)(x)(L+V)\left(\frac{\phi}{\Vert \phi\Vert}\right)(x)dx\,= \lambda_V.$$
By equation \eqref{GibbsPre}, this means that the Gibbs probability is the one that maximizes the pressure. In some sense, similar results are true for other settings, see \cite{DoVa,LMST,BCLMS,Ki2}.


\section{Time-reversal process and entropy production} \label{SEP}

In this section, we consider that the time parameter is bounded, $t\in [0,T]$ for a fixed $T>0$, to explore the time-reversal process. We will show that this time-reversal process is the jump process generated by the dual operator of $L$ in $\mathcal{L}^2(\mu)$, where $\mu$ is the invariant measure for $L^*$ defined in Section \ref{int}. Later, we study the properties of the entropy production rate, which can be used to describe the amount of work dissipated by an irreversible system. One can find related results in \cite{Wang,Da1,LM7,MN,MNS}.

Remember that the invariant measure satisfies $\mu(dx)=\theta(x)dx$ and that $L^{*} (\theta)=0$, where $L^{*}$ acts on $\mathcal{L}^2(dx)$. The substantial change from $\mathcal{L}^2(dx)$ to $\mathcal{L}^2(\mu)$ is that our reference measure, which was simply Lebesgue measure $dx$, becomes now $\theta(x)dx$. Taking that into account, the inner product in this new space is given by
$$\langle f,g \rangle_{\mu}=\int f(x)g(x)\mu(dx)=\int f(x)g(x)\theta(x)dx.$$

\begin{prop}
The dual operator of $L$ over $\mathcal{L}^2(\mu)$ is
$$ (\mathfrak L^{*}g)(x)= \int [g(y)-g(x)] \frac{\theta(y)}{\theta(x)} P(y,x)dy.$$
\end{prop}

\begin{proof}
To verify this, just compute
\begin{eqnarray*}
\langle Lf, g\rangle_{\mu}&=&\int (Lf)(x)g(x)\theta(x)dx\\
&=&\int\int g(x) \theta(x)P(x,y)f(y) dy dx - \int f(x)g(x) \theta(x)dx\\
&=&\int f(y) \int g(x)\theta(x)P(x,y) dx dy - \int f(x)g(x) \theta(x)dx\\
&=&\int f(z)\left(\int g(x)\frac{\theta(x)}{\theta(z)}P(x,z) dx\right) \theta(z)dz - \int f(z)g(z) \theta(z)dz\\
&=&\int f(z)\left(\int [g(x)-g(z)]\frac{\theta(x)}{\theta(z)}P(x,z) dx\right) \theta(z)dz\\
&=&\int f(z) (\mathfrak L^{*}g)(z) \theta(z)dz=\langle f, \mathfrak L^{*}g\rangle_\mu.
\end{eqnarray*}

In this computation, we use that $\int \frac{\theta(x)}{\theta(z)}P(x,z) dy =1$, which follows directly from equation \eqref{equacaodotheta}.
\end{proof}

Having discussed that, we now define the time-reversal process associated with the stationary Markov process $(X_t, \mu)$ and an interval of time $[0,T]$. The new process, denoted by $(\hat{X}_t)$, satisfies
$$ \E_{\mu} [\, g(\hat{X_0})f(\hat{X_t})]:=\E_{\mu}[\, g(X_T)f(X_{T-t})].$$

\begin{prop} The time-reversal process $\hat{X}_t$ has transition family equal to $P_t^*=e^{t\mathfrak{L}^*}$, the dual operator of $P_t$ over $\mathcal{L}^2(\mu)$.
\end{prop}

\begin{proof} Let $\hat{P}_t$ denote the transition family of $\hat{X}_t$. Using the Markov property and stationarity of the chain $X_t$, notice that this transition family satisfies, for all $ f,g \in \mathcal{L}^2(\mu)$,
\begin{eqnarray*}
\langle \hat{P}_tf,g\rangle_\mu &=& \int (\hat{P}_t f)(x)g(x) \ d\mu (x) = \E_{\mu} [ f(\hat{X}_{t})g(\hat{X}_0)]\\
&=&\E_{\mu} [ f(X_{T-t})g(X_T)] = \E_\mu\left[f(X_{T-t})\E_\mu[g(X_T)\vert \F_{T-t}]\right]\\
&=&\E_{\mu} [ f(X_0)\E_{X_0}[g(X_t)] ] = \int f(x) (P_t g)(x) d\mu(x)\\
&=&\langle f,P_t g\rangle_\mu.
\end{eqnarray*}

Since this is true for all $f,g \in \mathcal{L}^2(\mu)$, we get that $\hat{P}_t=P^{*}_t$. This also means that $\hat{L}=\mathfrak L^{*}$, where $\hat{L}$ is the infinitesimal generator of the semigroup $\hat{P}_t$.
\end{proof}

For a fixed $T>0$, we are interested in the relative entropy of $\hat{\p}_\mu$ with respect to $\p_\mu$, where $\hat{\p}_\mu$ is the probability induced on $\mathcal{D}$ by the time-reversal process with initial measure $\mu$. Notice that, by definition,
$$H_T(\p_\mu\vert\hat{\p}_{\mu})=-\int_D \log\left(\left.\frac{d\p_\mu}{d\hat{\p}_{\mu}}\right|_{\mathcal F_T}\right)(\omega)d\p_{\mu}(\omega).$$

Since, for the processes we are considering, we have $\lambda(x)=\hat{\lambda}(x)=1$ and $\hat{P}(x,dy)=\frac{\theta(y)}{\theta(x)}P(y,x)dy$. The Radon-Nikodym derivative (see \cite{AG}) implies that
$$\log\left(\left.\frac{d\p_\mu}{d\hat{\p}_{\mu}}\right|_{\mathcal F_T}\right)=\sum_{s\leq T}\log \left( \frac{P(X_{s-},X_s) \theta(X_{s-})}{P(X_s,X_{s-}) \theta(X_s)} \right)$$
and, consequently,
\begin{eqnarray*}
-H_T(\p_\mu\vert\hat{\p}_{\mu})&=& \E_{\mu} \left[ \sum_{s \leq T} \left\{\log \left( \frac{P(X_{s-},X_s)}{P(X_s,X_{s-}) } \right) + \log(\theta(X_{s^{-}})) -\log(\theta(X_s))\right\} \right]\\
&=&\E_\mu\left[\sum_{s \leq T} \log \left( \frac{P(X_{s-},X_s)}{P(X_s,X_{s-}) } \right) \right],
\end{eqnarray*}
because, for $\mu$ invariant, the telescopic summation
$$\E_{\mu} \left[ \sum_{s \leq T} \left\{\log(\theta(X_{s^{-}})) -\log(\theta(X_s))\right\} \right]=\E_\mu\left[\log(\theta(X_0))-\log(\theta(X_T))\right]=0.$$

In order to analyze the remaining term of this expression, we use the structure of the Markov process. Denoting by $0=T_0<T_1<\cdots$ the jump times of this process and by $\xi_n$ the value of the process on the interval $[T_{n-1},T_n)$, we have
\begin{eqnarray*}
-H_T(\p_\mu\vert\hat{\p}_{\mu})&=&\sum_{n=1}^\infty \E_\mu\left[\sum_{s \leq T} \log \left( \frac{P(X_{s-},X_s)}{P(X_s,X_{s-}) } \right)\1_{[T_n\leq T < T_{n+1}]} \right]\\
&=&\sum_{n=1}^\infty \E_\mu\left[\sum_{k=0}^{n-1} \log \left( \frac{P(\xi_k,\xi_{k+1})}{P(\xi_{k+1},\xi_k) } \right)\1_{[T_n\leq T < T_{n+1}]} \right]
\end{eqnarray*}

For simplicity, denote $\psi(x,y):=\log\left(\frac{P(x,y)}{P(y,x)}\right)$. Then,
\begin{eqnarray*}
-H_T(\p_\mu\vert\hat{\p}_{\mu})&=&\sum_{n=1}^\infty\sum_{k=0}^{n-1} \E_\mu\left[\psi(\xi_k,\xi_{k+1})\1_{[T_n\leq T < T_{n+1}]} \right]\\
&=&\sum_{n=1}^\infty\left(\E_\mu[\1_{[T_n\leq T < T_{n+1}]}]\sum_{k=0}^{n-1} \E_\mu[\psi(\xi_k,\xi_{k+1})]\right).
\end{eqnarray*}
In this computation, we use that the time variables $T_n$ (defined as the sum of $n$ independent exponential variables $\tau_k$ with parameter $1$) are independent of the spatial variables $\xi_k$. It is important to notice that this is not always true. In the general case, see \cite{KL}, each $\tau_k$ is distributed according to an exponential law of parameter $\lambda(\xi_k)$.

Now, we will analyze separately each expected value on the last expression. The first one is
\begin{eqnarray*}
\E_\mu[\1_{[T_n\leq T < T_{n+1}]}]&=&\int_0^\infty ds_0e^{-s_0}\cdots \int_0^\infty ds_n e^{-s_n}\left(\1_{[0\leq T-\sum_{i=0}^{n-1}s_i<s_n]}\right)\\
&=&e^{-T}\int_0^\infty\cdots\int_0^\infty ds_0\ldots ds_{n-1}\left(\1_{[\sum_{i=0}^{n-1}s_i\leq T]}\right)\\
&=&e^{-T}\frac{T^n}{n!},
\end{eqnarray*}
since the integrals can be recognized as a fraction (exactly $\frac{1}{2^n}$) of the volume of the ball in the $\R^n$ with 1-norm and radius $T$. For the second expected value, we use that $\mu$ is invariant for the chain to rewrite
$$\E_\mu[\psi(\xi_k,\xi_{k+1})]=\E_\mu[\psi(\xi_0,\xi_1)]=\int\mu(dx_0)\int P(x_0,x_1)\psi(x_0,x_1)dx_1,$$
which makes every term of the second sum equal. Then,
\begin{eqnarray*}
-H_T(\p_\mu\vert\hat{\p}_{\mu})&=&\sum_{n=1}^\infty e^{-T}\frac{T^n}{n!} \left(n \int\mu(dx_0)\int P(x_0,x_1)\psi(x_0,x_1)dx_1\right)\\
&=&T e^{-T}\sum_{n=1}^\infty\frac{T^{n-1}}{(n-1)!}\int\mu(dx_0)\int P(x_0,x_1)\psi(x_0,x_1)dx_1\\
&=&T \int\mu(dx_0)\int P(x_0,x_1)\psi(x_0,x_1)dx_1
\end{eqnarray*}

Using the above mentioned tools, we can now give meaning to the entropy production rate. This formulation, however, is not universal and depends on the physical system and its dynamical laws. Different formulations for entropy production are explored on \cite{LP}, where the authors made a review of the progress of these formulations. The point of view presented here relates to the one presented on \cite{BJPP}.

\begin{defi} The entropy production rate is defined as
$$ep := -H(\Prob_{\mu} \vert \hat{\Prob}_{\mu}) = -\lim_{T \to \infty} \frac{1}{T}\ H_T(\Prob_{\mu} \vert \hat{\Prob}_{\mu}).$$
\end{defi}

Using the computations we made before, it is possible to write the entropy production rate as
$$ep = \int \int \log\left(\dfrac{P(x,y)}{P(y,x)}\right) P(x,y) dy d\mu(x).$$

Notice that if we try to apply the concept of entropy production to a reversible process, satisfying $P(x,y)=P(y,x)$, we end up with $ep=0$.

\begin{prop}
For all transition functions $P(x,y)>0$, we have $ep\geq 0$.
\end{prop}

\begin{proof}
Since $\mathfrak{L}^{*}(1)=0$, we have $\int (Lf)(x) \,d\mu(x) = 0$ for every continuous function $f$. For $f=-\log \circ\, \theta$, we have that
$$ \int \int [\log(\theta(x))-\log(\theta(y))]P(x,y)dy d\mu(x)=0.$$
Therefore, we can add this term to the entropy production rate without changing its value:
\begin{eqnarray*}
ep&=&\int \int \log\left(\frac{\theta(x)P(x,y)}{\theta(y)P(y,x)}\right) P(x,y) dy d\mu(x)\\
&=&\int \int \left[\frac{\theta(x)P(x,y)}{\theta(y)P(y,x)}\right]\log\left(\frac{\theta(x)P(x,y)}{\theta(y)P(y,x)}\right) \frac{\theta(y)}{\theta(x)}P(y,x) dy d\mu(x).
\end{eqnarray*}

Since $\int \int \frac{\theta(y)}{\theta(x)}P(y,x) dy d\mu(x) = 1$, we can use this as a probability measure in order to apply the Jensen inequality for the convex function $\psi(z)= z \log z $. In this way,
\begin{eqnarray*}
ep&=&\int \int \psi\left(\frac{\theta(x)P(x,y)}{\theta(y)P(y,x)}\right) \frac{\theta(y)}{\theta(x)}P(y,x) dy d\mu(x)\\
&\geq&\psi\left(\int\int \left[\frac{\theta(x)P(x,y)}{\theta(y)P(y,x)}\right]\frac{\theta(y)}{\theta(x)}P(y,x) dy d\mu(x)\right)\\
&=&\psi\left(\int\int P(x,y)dyd\mu(x)\right)=\psi(1)=0.
\end{eqnarray*}
\end{proof}

The idea of this proof was similar to the one in Lemma 3.3 in \cite{PP}.

\begin{theo}\label{estamos} The entropy production rate of the time reversal process is the same as the original process:
$$ep^{*}:=-H(\hat{\mathbb{P}}_{\mu} \vert \mathbb{P}_{\mu})=ep.$$
\end{theo}

\begin{proof} Since $L(1)=0$, we have that $\int (\mathfrak{L}^{*} g)(x) \,d\mu(x) = 0$ for every continuous function $g$, For $g=\log\circ\,\theta^2$, we have that
$$ \int \int [\log(\theta^2(y))-\log(\theta^2(x))]P^*(x,y)dy d\mu(x)=0,$$
where $P^{*}(x,y)=\frac{\theta(y)}{\theta(x)}P(y,x)$. One can see that
\begin{eqnarray*}
ep^* &=& \int \int \log\left(\dfrac{P^*(x,y)}{P^*(y,x)}\right) P^*(x,y) dy d\mu(x)\\
&=& \int \int \log\left(\dfrac{\theta^2(y) P(y,x)}{\theta^2(x) P(x,y)}\right) P^*(x,y) dy d\mu(x)\\
&=& \int \int \log\left(\dfrac{P(y,x)}{ P(x,y)}\right) P^*(x,y) dy d\mu(x)\\
&=& \int \int \log\left(\dfrac{P(y,x)}{ P(x,y)}\right) \frac{\theta(y)}{\theta(x)} P(y,x) dy [\theta(x) dx]\\
&=& \int \int \log\left(\dfrac{P(y,x)}{ P(x,y)}\right) P(y,x) d\mu(y) dx\\
&=& \int \int \log\left(\dfrac{P(x,y)}{ P(y,x)}\right) P(x,y) dy d\mu(x) = ep.
\end{eqnarray*}
\end{proof}

\section{Expansiveness of the semi-flow $\Theta_t $ on $\mathcal{D}$} \label{SIK}
In this section, we consider an extended Skorohod space  $\hat{\mathcal{D}} $ of the \textit{c\`adl\`ag} paths $w: \R \to [0,1]$ and $\hat{\Theta}_t$, $t \in \mathbb{R},$ the bidirectional flow on $\hat{\mathcal{D}}$, acting on $w$ by translation to the left: $(\hat{\Theta}_t w)(s)= w(s+t).$ One can show that two paths on $\mathcal{D}$ that coincide up to time $t$ have a distance between them, using the  Skorokhod metric defined below, limited by $e^{-t}$. This means that given two paths of such type, it is possible to increase the distance by applying $\Theta_t$.

Let $\Lambda$ be the set of continuous functions $f$ such that
$$\gamma(\lambda):=\supess_{t \ge 0} |\log \lambda '(t)|<\infty$$
Moreover, recall the definition of the Skorokhod distance (see \cite{EK}):
$$ d(x,y) = \inf_{\lambda \in \Lambda} \left[ \gamma(\lambda) \vee \int_0^{\infty} e^{-u}d(x,y,\lambda,u)du  \right].$$

Let $\mathcal{D}^*$ be the set of paths $w:(-\infty,0]\to [0,1]$ continuous at left and with a limit at right. We can denote a typical path $w$ in $\hat{\mathcal{D}}$ as
$$w(s)=(w_1\vert w_2)(s) = \left\{\begin{array}{ll} w_1(s), & \textrm{ for }s<0, \\ w_2(s), & \textrm{ for } s\geq 0,\end{array}\right. $$
where $w_1\in \mathcal{D}^*$ and $w_2\in \mathcal{D}$. In this way, we can identify $\hat{\mathcal{D}}\mapsto \mathcal{D}^*\times D$ and define the projections $\Pi_1(w)=w_1$ and $\Pi_2(w)=w_2$. By convention, we will always use the time $t=0$ to set this.

From two paths $w_1\in \mathcal{D}^*$ and $w_2\in \mathcal{D}$, we can go to $\hat{\mathcal{D}}$ by $( w_1\vert w_2)$, then apply $\hat{\Theta}_{-t}$ and go back to $\mathcal{D}$ using $\Pi_2$. By doing this, we ended up with
$$\Pi_2(\hat{\Theta}_{-t}( w_1\vert w_2 ))(s) = (w_1\vert^t\, w_2)(s) := \left\{\begin{array}{ll} w_1(s-t), & \textrm{ for }s<t, \\ w_2(s-t), & \textrm{ for } s\geq t,\end{array}\right. $$
defined for $s\geq 0$, as Figure \ref{fig:paths} shows.
\begin{figure}[h]
\center
\includegraphics[scale=0.35]{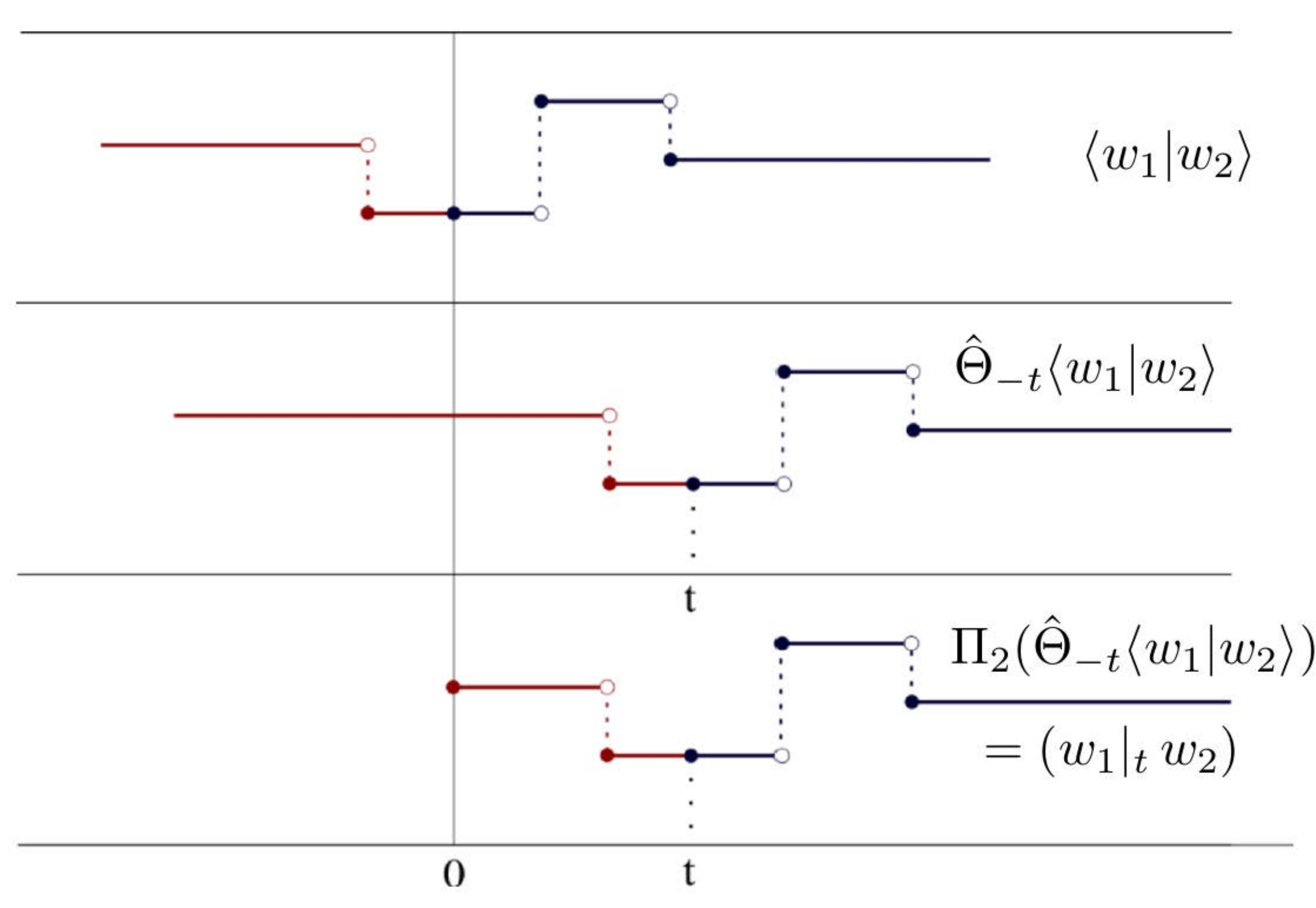}
	\caption{The bilateral shift and the projection $\Pi_2$}
	\label{fig:paths}
\end{figure}

\begin{theo}\label{excam} The continuous-time shift $\Theta_t$, acting on the  Skorokhod space $\mathcal{D}$, is expanding: given paths $w_1\in \mathcal{D}^*$ and $w_2, w'_2\in \mathcal{D}$, for all $t\geq 0$,
\begin{equation} \label{expand}
d\left((w_1|^t\, w_2),(w_1|^t\, w'_2)\right)\leq\int_t^\infty e^{-u}du=e^{-t}.
\end{equation}
\end{theo}

\begin{proof}

Fix $I$ as the identity function. Then, $\gamma(I)=0$ and
\begin{eqnarray*}
d((w_1|^t\, w_2),(w_1|^t\, w'_2)) &\le& \int_0^\infty e^{-u} d((w_1|^t\, w_2),(w_1|^t\, w'_2),I,u) du\\
&=&\int_0^\infty e^{-u} \sup_{s\geq 0} q\left((w_1|^t\, w_2)(s\wedge u),(w_1|^t\, w'_2)(s\wedge u)\right)du,
\end{eqnarray*}
where $q=r\wedge 1$ with $r$ denoting the (Lebesgue) metric on the state space $[0,1]$.

For $u<t$, the distance $q$ above is $q(w_1(t-s\wedge u),w_1(t-s\wedge u))=0$. Otherwise, the distance $q$ is upper bounded by $1$. Then,
$$d\left((w_1|^t\, w_2),(w_1|^t\, w'_2)\right)\leq\int_t^\infty e^{-u}du=e^{-t}.$$
\end{proof}

\medskip

\appendix

\section{Appendix 1 - Existence of $K_t(x,y)$}\label{A1} 

In this section, we will show explicitly the existence of a function $K_t(x,y)$, which has a relation with the semigroup $e^{tL}$ given by equation \eqref{functK}. We can write $L=\mathcal{L} - I$, where $\mathcal{L}$ is acting on functions as $(\mathcal{L}f)(x)=\int f(y)P(x,y)dy$. One can write down the action of the powers $L^k$ which appear in $e^{tL}$ in a simple way using the Newton binomial, since $\mathcal{L}$ and $-I$ commute:
$$ L^k= (\mathcal{L}-I)^k = \sum_{j=0}^k {k \choose j} \mathcal{L}^j (-I)^{k-j} = \sum_{j=0}^k (-1)^{k-j} {k \choose j} \mathcal{L}^j,$$
where $\mathcal{L}^0(f)=I(f)=f$. To go further, we need to consider the following transition functions: for all $k\ge 2$,
$$ P^k(x,y):=\int \cdots \int P(x,z_1)P(z_1,z_2)\cdots P(z_{k-1},y) dz_1 dz_2... dz_{k-1}.$$

Of course, $P^1(x,y)=P(x,y)$ and $P^{k+1}(x,y) = \int P^k(x,z)P(z,y)dz$. Now, we state that
$$(\mathcal{L}^kf)(x)=\int f(y)P^k(x,y)dy,$$
for every $k\ge 1$. To verify this, one can use induction:
\begin{eqnarray*}
(\mathcal{L}^{k+1}f)(x)&=&\mathcal{L}^{k}(\mathcal{L}f)(x)\\
&=&\int (\mathcal{L}f)(y)P^k(x,y)dy\\
&=&\int \int f(z)P(y,z)dz \, P^k(x,y)dy\\ 
&=&\int f(z) \int P^k(x,y)P(y,z)dy \ dz\\
&=&\int f(z) P^{k+1}(x,z) dz.
\end{eqnarray*}
Above, to change the order of integration, we use the continuity of $P$ and $f$ over the compact state space or the continuity of $P$ and the boundedness of $f$ to assure that the integral is finite. 

Now, we can compute $L^k$:
\begin{eqnarray*}
(L^kf)(x)&=&\sum_{j=0}^k (-1)^{k-j} {k \choose j} (\mathcal{L}^jf)(x)\\
&=&(-1)^k f(x) + \sum_{j=1}^k (-1)^{k-j} {k \choose j} \int f(y)P^j(x,y)dy.
\end{eqnarray*}
Changing the order of terms, we get
\begin{eqnarray*}
(L^kf)(x)&=&(-1)^k f(x) + \int f(y) \left[\sum_{j=1}^k (-1)^{k-j} {k \choose j} P^j(x,y) \right] dy\\
&=&(-1)^k f(x) + \int f(y) \ Q_k(x,y)dy,
\end{eqnarray*}
where $Q_k(x,y)$ is the expression inside the bracket. Notice that
$$K_t(x,y)=\sum_{k=1}^\infty \frac{t^k}{k!} \ Q_k(x,y)$$
is our desired function, because
\begin{eqnarray*}
(e^{tL}f)(x)&=&f(x) + \sum_{k=1}^{\infty} \frac{t^k}{k!}(L^kf)(x)\\
&=& f(x) + \sum_{k=1}^{\infty} \frac{t^k}{k!} \left[ (-1)^kf(x) + \int f(y)  Q_k(x,y) dy \right] \\
&=& f(x) \sum_{k=0}^{\infty} \frac{(-t)^k}{k!} + \int f(y) \ \sum_{k=1}^{\infty} \frac{t^k}{k!}\ Q_k(x,y)\ dy \\
&=& f(x)e^{-t} + \int f(y) K_t(x,y) \ dy.
\end{eqnarray*}

Considering the dynamics involved, the first term, which cannot be merged into $K_t(x,y)$, corresponds to the probability of not observing any jump in the interval $[0,t]$.

\subsection{Properties of $K_t(x,y)$}

We denote by $P_t=e^{tL}$. Then, we calculate 
$$\partial_t (P_tf)(x)=-e^{-t}f(x)+\int f(y)(\partial_tK_t)(x,y)dy$$
and
\begin{eqnarray*}
L(P_tf)(x)&=&\int (P_tf)(y)P(x,y)dy - (P_tf)(x) \\
&=&  \int e^{-t} f(y)P(x,y)dy + \int \int f(z)K_t(y,z)dz  P(x,y) dy - (P_tf)(x) \\
&=& \int e^{-t} f(y)P(x,y)dy + \int f(z) \left( \int P(x,y)K_t(y,z)  dy \right) dz \\*
&&-e^{-t}f(x)-\int f(y)K_t(x,y)dy.
\end{eqnarray*}
Reordering the terms, we conclude that $L(P_tf)(x)$ is equal to
$$-e^{-t}f(x) + \int f(y) \left(-K_t(x,y)+e^{-t}P(x,y)+\int  P(x,z)K_t(z,y)dz  \right) dy.$$

As $P_t$ is the homogeneous semigroup generated by the infinitesimal generator $L$, the Kolmogorov equations imply that $L(P_tf) = \partial_t P_t(f) = P_t(Lf).$ From this, we conclude the equality of these two expressions for every $f$. Then,
$$\partial_t K_t(x,y) = - K_t(x,y) + e^{-t}P(x,y)+\int P(x,z)K_t(z,y)dz.$$

The above is equal to
$$\partial_t K_t(x,y) = L(K_t(\cdot,y))(x) + e^{-t}P(x,y)$$
and, if we write down the other equation $\partial_t P_tf=P_t(Lf)$, the only change is the last integral for $\int P(z,y)K_t(x,z) dz$, which results in
$$\partial_t K_t(x,y) = L^{*}(K_t(x,\cdot))(y) + e^{-t}P(x,y).$$

\medskip

Another way to explore $K_t(x,y)$ is looking to the property of semigroup: $P_s\circ P_t=P_{s+t}$. This leads us to
$$(P_{s+t}f)(x)=(e^{(s+t)L}f)(x)=e^{-(s+t)}f(x)+\int f(y)K_{s+t}(x,y)dy,$$
while $P_{t}(P_sf)(x)$ is equal to
\begin{eqnarray*}
&&e^{tL}(e^{sL}f)(x)\\
&=& e^{-t}(e^{sL}f)(x)+\int (e^{sL}f)(y)K_t(x,y)dy\\
&=& e^{-t}\left[e^{-s}f(x)+\int f(y)K_s(x,y)dy \right]\\
&&+ \int \left[e^{-s}f(y)+\int f(z)K_s(y,z)dz \right]K_t(x,y)dy \\
&=& e^{-(t+s)}f(x) + \int f(y)\left(e^{-t}K_s(x,y) + e^{-s}K_t(x,y)\right)dy\\
&&+ \int \int f(z)K_s(y,z)K_t(x,y)dzdy\\
&=&\int f(y)\left(e^{-t}K_s(x,y)+ e^{-s}K_t(x,y) + \int K_t(x,z)K_s(z,y)dz \right)dy\\
&&+e^{-(t+s)}f(x).
\end{eqnarray*}

This means
$$K_{s+t}(x,y)=e^{-t}K_s(x,y)+e^{-s}K_t(x,y)+\int K_t(x,z)K_s(z,y)dz.$$

Notice that the last equation is the expression (1.3.1) in \cite{BGL} for our transition function $p_t (y,dx) = K_t (x,y)dx+ e^{-t} \delta_y(dx)$.


\section{Appendix 2 - Existence of $K^V_t$} \label{A2}

In this section, we will show explicitly the existence of a function $K^V_t(x,y)$ which has a relation with the semigroup $e^{t(L+V)}$ given by equation \eqref{functKV}. Here, we are considering a general infinitesimal generator $L$ of a Markov jump process. We will analyze the equation \eqref{expLV} in terms of the graphic construction of the jump process $(X_t)$, i.e., we will use that the trajectories are piecewise constants:
$$\E_x \left[ \ e^{\int_0^t V(X_r)dr} f(X_t)\  \right] = \sum_{n=0}^\infty \E_x \left[ \ e^{\int_0^t V(X_r)dr} f(X_{T_n}) \1_{[T_n \leq t < T_{n+1}]} \ \right],$$
where $0=T_0<T_1<T_2<\cdots$ are the times that $X_t$ jumps.

The $n=0$ term of this sum represents the time before the first jump. In this case, we have $s<T_1$ and the process $X_s\equiv x$. Then, this first term is equal to 
$$e^{tV(x)}f(x)\p_x[\tau_0>t]=e^{tV(x)}f(x)e^{-t\lambda(x)},$$
where $\tau_0$ is a random variable with exponential distribution of parameter $\lambda(x)$.

For the terms $n\geq 1$, we need further analysis. For each $k$, set $x_k=X_{T_k}$ and let $\tau_k$ be a exponential random variable with parameter $\lambda(x_k)$. By this, under $\1_{[T_n \leq t < T_{n+1}]}$, we have
$$\int_0^t V(X_r)dr = \sum_{i=0}^{n-1} \tau_iV(x_i) + \left(t-\sum_{i=0}^{n-1} \tau_i\right)V(x_n).$$
Now, define
$$ \varphi^{n,V}_t(x_0,...,x_n)= \exp\left[\sum_{i=0}^{n-1} \tau_iV(x_i) + \left(t-\sum_{i=0}^{n-1} \tau_i\right)V(x_n) \1_{[T_n \le t < T_{n+1}]}\right].$$
Notice that, for a fixed $t$, all functions  $\varphi^{n,V}_t$ are null except for the one whose $n$ is equal to the number of jumps until time $t$. In this way, using the kernel $P(x,dy)=P(x,y)dy$, the $n$th term of the summation becomes
$$ \int \cdots \int \varphi^{n,V}_t(x_0,...,x_n) f(x_n) P(x_0,x_1)dx_1\cdots P(x_{n-1},x_n)dx_n.$$
This expression is equal to $\int Q^{n,V}_t(x,x_n) f(x_n) dx_n$ if we define
$$Q^{n,V}_t(x,x_n)= \int \cdots \int \varphi^{n,V}_t(x_0,...,x_n) P(x_0,x_1)\cdots P(x_{n-1},x_n) dx_1\cdots dx_{n-1}.$$

Finally,
\begin{eqnarray*}
 \left(e^{t(L+V)}f\right)(x) &=& e^{tV(x)}f(x)e^{-t\lambda(x)} + \sum_{n=1}^\infty \int Q^{n,V}_t(x,x_n) f(x_n) dx_n \\
&=& e^{t(V(x)-\lambda(x))} f(x) + \int \, \sum_{n=1}^{\infty} Q^{n,V}_t(x,y) f(y) dy \\
&=& e^{t(V(x)-\lambda(x))} f(x) + \int \,K^V_t(x,y) f(y) dy,
\end{eqnarray*}
where $K^V_t(x,y)=\sum\limits_{n=1}^{\infty} Q^{n,V}_t(x,y)$. Notice that $Q_t^{n,V}(x,y)\geq 0$ for all $t$. Furthermore, it is strictly positive when $n$ equals the number of jumps until time $t$. Then, $K^V_t(x,y)>0$, for every $x,y \in [0,1]$.

\subsection{Properties of $K^V_t$}

Now, we proceed in the same way we have done with $K_t$, looking for a differential equation that $K^V_t$ satisfies, in the case of $\lambda \equiv 1$. For the semigroup $P^V_t=e^{t(L+V)}$, we have $(L+V)(P^V_t f)=\partial_t P^V_t (f)=P^V_t((L+V)f)$. The middle term opens as
$$\partial_t P^V_t(f)(x) = (V(x)-1)e^{tV(x)-t}f(x) + \int \partial_t K^V_t(x,y)f(y)dy$$
while the last term is
\begin{eqnarray*}
&&P^V_t((L+V)f)(x)\\*
&=&e^{tV(x)-t}(L+V)(f)(x)+\int K^V_t(x,y)(L+V)(f)(y)dy\\
&=&e^{tV(x)-t}\left[\int P(x,y)f(y)dy + (V(x)-1)f(x)\right]\\
&&+\int K^V_t(x,y)(L+V)(f)(y)dy.
\end{eqnarray*}
We get that, for every $f$,
$$ \int \partial_t K^V_t(x,y)f(y)dy = e^{tV(x)-t}\int P(x,y)f(y)dy+\int K^V_t(x,y)(L+V)(f)(y)dy.$$

Using the definition of $L+V$, we can make a computation to rewrite the right-hand side of the above equation as 
$$\int\left[e^{tV(x)-t}P(x,y)+\int K_t^V(x,z)P(z,y)dz+K_t^V(x,y)(V(y)-1)\right]f(y)dy,$$
which means that
\begin{eqnarray*}
\partial_t K^V_t(x,y)&=&e^{tV(x)-t}P(x,y)+\int K_t^V(x,z)P(z,y)dz+K_t^V(x,y)(V(y)-1)\\
&=&e^{tV(x)-t}P(x,y) + (L^*+V)(K_t^V(x,\cdot))(y).
\end{eqnarray*}

Similarly, if we open the other equation $(L+V)(P^V_t f)=\partial_t P^V_t (f)$, we conclude
$$\partial_t K^V_t(x,y) = e^{tV(x)-t}P(x,y) + (L+V)(K_t^V(\cdot,y))(x).$$


\section{Appendix 3 - Another look of Feynman-Kac formula for symmetrical $L$}\label{RuelleKac}

Consider $X_t$ a continuous-time process with state space $[0,1]$ and infinitesimal generator $L$. 
Let $\mu$ a measure on $[0,1]$ such that the infinitesimal generator $L$ is selfadjoint in $\mathcal L^2(\mu)$, that is,    \begin{equation}\label{sim2}\int_{[0,1]}(Lf)(x)g(x)\,d\mu(x)=\int_{[0,1]}f(x)(Lg)(x)\,d\mu(x), \end{equation}
for all continuous functions $f,g:[0,1]\to \mathbb R$. Consider $\mathbb P_{\mu}$ the {\it apriori}
probability, which is induced by the initial measure $\mu$ and the infinitesimal generator $L$. And. denote by $\mathbb E_{\mu}$ the expectation concerning to the $\mathbb P_{\mu}$. Let  $V$ be a bounded measurable function on $[0,1]$ taking values on $\mathbb{R}$. Then, it is possible to show that, for any fixed $T>0$, we have
  \begin{equation}\label{sim3}
  \begin{split}
      \int_{\mathcal{D}}e^{\int_0^{t} V(w(r))\,dr}& f(w(t))\,g(w(0))\,d\mathbb P_{\mu}(w) \\&=\int_{\mathcal{D}} e^{\int_0^{t} V(w(r))\,dr} f(w(0))\,g(w(t))\,d\mathbb P_{\mu}(w) \,,
  \end{split}
\end{equation}
for all continuous functions $f,g:[0,1]\to \mathbb R$.

We start the proof of \eqref{sim3} by observing that it can be rewritten as 
    \begin{equation}\label{sim1}\mathbb E_{\mu} \Big[e^{\int_0^{t} V(X_r)\,dr} f(X_t)\,g(X_0)\Big] = \mathbb E_{\mu} \Big[e^{\int_0^{t} V(X_r)\,dr} f(X_0)\,g(X_t)\Big] \,.
\end{equation} Then we will prove \eqref{sim1},  to do this we 
    use \begin{equation}\label{explain_measure}
    \mathbb P_\mu[A]=\int_{[0,1]} \mathbb P_{x}[A]\, d\mu(x),
\end{equation} in the left-hand side of \eqref{sim1}, and we obtain
    \begin{equation}\label{Aa1}
    \begin{split}
        \mathbb E_{\mu} \Big[e^{\int_0^{t} V(X_r)\,dr} f(X_t)\,g(X_0)\Big] &= 
  \int_{[0,1]} \mathbb E_{x} \Big[e^{\int_0^{t} V(X_r)\,dr} f(X_t)\,g(X_0)\Big]\,d\mu(x)\\
  &= 
  \int_{[0,1]} \mathbb E_{x} \Big[e^{\int_0^{t} V(X_r)\,dr} f(X_t)\Big]\,g(x)\,d\mu(x)\\
  &=\int_{[0,1]}(P^V_tf)(x)\, g(x)\,d\mu(x)\,,
    \end{split}
    \end{equation}
    where the last equality is due to the expression 
    \begin{equation}\label{0}
P_{t}^V (f)(x)\,:=\, \mathbb E_{x} \big[e^{\int_0^{t} V(X_r)\,dr} f(X_t)\big]\,,
\end{equation} 
where  $P^V_t$ is the semigroup associated to the infinitesimal generator $L+V$. By \eqref{sim2}, we have that $L+V$ is selfadjoint, that is, 
    \begin{equation*}
   \begin{split}
       & \int_{[0,1]}(L+V)(f)(x)\;g(x)\,d\mu(x)= \int_{[0,1]}(Lf)(x)g(x)\,d\mu(x)+\int_{[0,1]}V(x)f(x)g(x)\,d\mu(x)\\&=\int_{[0,1]}f(x)\;(L+V)(g)(x)\,d\mu(x).
   \end{split}
    \end{equation*}
    Then the semigroup $P^V_t$ is associated with $L+V$ and it is selfadjoint too. Thus,
      \begin{equation}\label{Aa2}
    \begin{split}
\int_{[0,1]}(P^V_tf)(x)\, g(x)\,d\mu(x)=\int_{[0,1]}f(x)\,(P^V_tg)(x)\,d\mu(x)\,,
    \end{split}
    \end{equation}
    for all continuous functions $f,g:[0,1]\to \mathbb R$. Writing  the semigroup $P^V_t$ with the expression \eqref{0} and  using \eqref{explain_measure}, we get
      \begin{equation}\label{Aa3}
    \begin{split}
\int_{[0,1]}f(x)\,(P^V_tg)(x)\,d\mu(x)&=\int_{[0,1]}f(x)\,\mathbb E_{x} \Big[e^{\int_0^{t} V(X_r)\,dr} g(X_t)\Big]\,d\mu(x)\\
&=\int_{[0,1]}\mathbb E_{x} \Big[e^{\int_0^{t} V(X_r)\,dr} g(X_t)\, f(X_0)\Big]\,d\mu(x)\\
&=\mathbb E_{\mu} \Big[e^{\int_0^{t} V(X_r)\,dr} g(X_t)\, f(X_0)\Big]\,.
    \end{split}
    \end{equation}
    Putting \eqref{Aa1}, \eqref{Aa2}, and \eqref{Aa3} together we obtain \eqref{sim1}.

To conclude this Appendix section we use the expression \eqref{sim3} with a smooth function $g_n$, and 
taking $g_n(y)\to \textbf{1}_{x}(y)$, we get 
    \begin{equation*}
   \begin{split}
        \int_{\mathcal D}e^{\int_0^{t} V(w(r))\,dr} &f(w(t))\,\textbf{1}_{w(0)=x}\,d\mathbb P_{\mu}(w) \\=&\int_{\mathcal D} e^{\int_0^{t} V(w(r))\,dr} f(w(0))\,\textbf{1}_{w(t)=x}\,d\mathbb P_{\mu}(w) \,,
   \end{split}
\end{equation*}for all continuous function $f:[0,1]\to \mathbb R$ or using another notation
    \begin{equation*}
   \begin{split}
     \mathbb E_{\mu} \Big[e^{\int_0^{t} V(X_r)\,dr} \, f(X_t) \,\textbf{1}_{X_0=x}\Big]=
        \mathbb E_{\mu} \Big[e^{\int_0^{t} V(X_r)\,dr} \, f(X_0) \,\textbf{1}_{X_t=x}\Big]\,.
   \end{split}
\end{equation*}

\bigskip

Emails: 

$jojoknorst@gmail.com$

$arturoscar.lopes@gmail.com$

$gustavo.muller\_ nh@hotmail.com$ 

$neumann.adri@gmail.com$

\medskip



\begin{thebibliography}{99}
\footnotesize


\bibitem{LNP1}  P. Adamo, R. Belousov, and L. Rondoni,
Fluctuation-Dissipation and Fluctuation
Relations: From Equilibrium to Nonequilibrium
and Back. \emph{Large Deviations in Physics},
Lect. Notes in Physics, 92-133
Vol. 885, Springer Verlag (2014)

\bibitem{BEL} A. Baraviera, R. Exel and A. Lopes, A Ruelle Operator for continuous-time Markov
chains.  \emph{ S\~ao Paulo Journal of Mathematical Sciences} Vol 4 n. 1, pp 1-16 (2010).


\bibitem{BCLMS} A. Baraviera, L. M. Cioletti,  A. O. Lopes,
J. Mohr and R. R. Souza, On the general $XY$ Model: positive and zero temperature, selection and non-selection, \emph{Rev. in Math Phys}, Vol. 23, N. 10, pp 1063-1113 (2011).


\bibitem{BGL}
D. Bakry, I. Gentil and M. Ledoux,
Analysis and
Geometry of
Markov Diffusion
Operators, Springer

\bibitem{BJPP} T. Benoist, V. Jaksic, Y. Pautrat and C-A. Pillet.
On entropy production of repeated quantum measurements I. General Theory.
\emph{Communications in Mathematical Physics}
Vol. 357, pp. 77-123 (2018).

\bibitem{Bo} A. Bobrowski,
Functional Analysis for Probability and Stochastic Processes. Cambridge Press (2005)



\bibitem{BKL} J. E. Brasil, J. Knorst and A. O. Lopes,
Thermodynamic formalism for continuous-time quantum Markov semigroups: the detailed balance condition, entropy, pressure and equilibrium quantum processes,  \emph{Open Systems and Information Dynamics} Volume 30, Issue 04 2350018 (13 pages) (2023)






\bibitem{DoVa}
M. D. Donsker and S. R. S. Varadhan,
On a Variational Formula for the Principal Eigenvalue for Operators with Maximum Principle.
\emph{Proc. Nat. Acad, Scien}, March 1, 1975 72 (3) 780--783



\bibitem{De}
K. Deimling. Nonlinear Functional Analysis, Springer Verlag, 1985


\bibitem{EK}
S. N. Ethier and T. G. Kurtz, Markov processes - Characterization
and Convergence, Wiley (2005)


\bibitem{JK} J. Feng and T. Kurtz, Large Deviations for Stochastic Processes, AMS

\bibitem{GaCo} 
G. Gallavotti and E. G. D. Cohen. Dynamical Ensembles in Nonequilibrium 
\emph{Statistical Mechanics, Phys. Rev. Lett.} 74, 2694 (1995).

\bibitem{Ga} G. Gallavotti,
Fluctuation patterns and conditional reversibility in nonequilibrium systems
\emph{Annales de l'I.H.P. Physique theorique} (1999)
Volume: 70, Issue: 4, page 429--443


\bibitem{Gomes} 
D. A. Gomes, A stochastic analogue of Aubry-Mather theory. \emph{Nonlinearity}
15, 581--603, 2002

\bibitem{Da1}
Da-Quan Jiang
Min Qian
Min-Ping Qian,
Mathematical Theory
of nonequilibrium
Steady States, Lect. Notes in Math 1833 (2004)


\bibitem{Da} Da-quan Jiang, Min Qian and Min-ping Qian. Entropy Production and
Information Gain in Axiom-A \emph{Systems. Commun. Math. Phys.} 214, 
389--409 (2000).


\bibitem{K} M. Kac, Integration in some function spaces and some of its application, Acad Naz dei Lincei
Scuola Superiore Normale Superiore, Piza, Italy (1980)

\bibitem{Ki1} Y. Kifer, Large Deviations in Dynamical Systems and Stochastic processes, \emph{TAMS} Vol 321, N.2, 505--524 (1990)


\bibitem{Ki2} Y. Kifer, Principal eigenvalues, topological pressure, and stochastic stability of equilibrium states, \emph{Israel Journal of Mathematics}
Vol. 70, No. I, pp 1--47 (1990)

\bibitem{KL} C. Kipnis and C. Landim.
Scaling Limits of Interacting Particle Systems.
Springer Science \& Business Media (1998).

\bibitem{LP} G. T. Landi and M. Paternostro.
Irreversible entropy production: From classical to quantum.
\emph{Reviews of Modern Physics}.
Vol. 93, no. 3 (2021).

\bibitem{Li} T. Liggett, Continuous Time Markov Processes, AMS, (2010).







\bibitem{LT} A. O. Lopes and Ph. Thieullen,
Transport and large deviations for Schrodinger
operators and Mather measures,  \emph{Modeling, Dynamics, Optimization and Bioeconomics III} Springer Verlag, 247-255 (2018)


\bibitem{LMST} A. O. Lopes, J. Mohr, R. Souza and Ph. Thieullen,
Negative Entropy, Zero temperature and stationary Markov chains on the interval, \emph{Bull. Soc. Bras. Math.} Vol 40 n 1, (2009), 1-52

\bibitem{LN} A. O. Lopes, G. Muller and A. Neumann,
Diffusion Processes: entropy, Gibbs states  and  the continuous time Ruelle  operator, preprint (2023)

\bibitem{LN1} A. O. Lopes and A. Neumann,
Large Deviations for stationary probabilities of a family of continuous time Markov chains via Aubry-Mather theory, \emph{Journal of Statistical Physics} Vol. 159 -- Issue 4 pp 797-822 (2015)





\bibitem{LMMS}
A. O. Lopes, J. K. Mengue, J. Mohr and R. R. Souza, Entropy and Variational Principle for one-dimensional Lattice Systems with a general a-priori probability: positive and zero temperature, \emph{Erg. Theory and Dyn Systems}, 35 (6), 1925--1961 (2015)

\bibitem{LM7} A. O. Lopes and J. K. Mengue,
On information gain, Kullback-Leibler
divergence, entropy production and the
involution kernel, \emph{Disc. and Cont. Dyn. Syst. Series A} Vol. 42, No. 7, 3593--3627 (2022)


\bibitem{MNS} C. Maes, K. Netocny, and B. Shergelashvili,
A selection of nonequilibrium issues, \emph{Methods of Contemporary Mathematical Statistical
Physics} 247-306 (2009)

\bibitem{MN} C. Maes and K. Netocny,
Time-Reversal and Entropy, \emph{Journal of Statistical Physics} Vol. 110, Nos. 1/2, (2003)

\bibitem{Morales} C. A. Morales,
Bowen-Walters expansiveness for semigroups of linear operators,  \emph{Ergodic Theory and Dynamical Systems}  Volume 43 , Issue 6 , June 2023, pp. 1942-1951



\bibitem{MT} S. Meyn and R. Tweedie,
Markov Chains and Stochastic Stability, Springer (1993)


\bibitem{AG} G. Muller and A. Neumann; Some general results for continuous-time Markov chain, preprint (2023).


\bibitem{PP}
W. Parry and M.  Pollicott. Zeta functions and the periodic
orbit structure of hyperbolic dynamics, \emph{Ast\'erisque}
Vol {187-188} (1990)



\bibitem{LNP2} 
A. Politi,
Stochastic Fluctuations in Deterministic Systems, \emph{Large Deviations in Physics},
Lect. Notes in Physics, 243-261
Vol. 885, Springer Verlag (2014)

\bibitem{Str}
W. Strock, An introduction to the theory of large deviations, Springer


\bibitem{JQQ} Da-Quan Jiang,
Min Qian,
Min-Ping Qian,
Mathematical Theory
of nonequilibrium,
Steady States, Lect notes in Math 1833

\bibitem{Ru0} D. Ruelle; Positivity of entropy production in nonequilibrium statistical mechanics, \emph{Journal of Statistical Physics} Vol. 85, 1996, 1996.3, pp. 1--25

\bibitem{Ru} D. Ruelle; A generalized detailed balance relation. \emph{Journal of Statistical Physics} 164,
no. 3, 463--471 (2016)

\bibitem{Wang} Y. Wang and H. Qian,
Mathematical Representation of Clausius' and Kelvin's
Statements of the Second Law and Irreversibility, \emph{Journal of Statistical Physics} (2020) 179:808--837


\end{thebibliography}
\end{document}